\documentclass[EJP]{ejpecp} 

\usepackage{verbatim,graphicx, subfigure,longtable}
\usepackage{amssymb,amsthm,amscd}
\usepackage{pstricks}
\usepackage{epstopdf}
\usepackage{color}
\usepackage{transparent}

\SHORTTITLE{Uniform in time particle approximations for PKS} 

\TITLE{Uniform in time interacting particle approximations for nonlinear equations of Patlak-Keller-Segel type
	\thanks{Research partially supported by National Science Foundation (DMS-1305120), the Army Research Office (W911NF-14-1-0331) and DARPA (W911NF-15-2-0122).}} 

\AUTHORS{%
  Amarjit Budhiraja \quad and \quad Wai-Tong (Louis) Fan}


\AMSSUBJ{60K35; 60H30; 60H35; 60K40; 60F05} 

\SUBMITTED{July 24, 2016} 
\ACCEPTED{January 8, 2017} 

\VOLUME{22}
\YEAR{2016}
\PAPERNUM{8}
\DOI{}

\ABSTRACT{We study a system of interacting diffusions that models chemotaxis of biological cells or microorganisms (referred to as particles) in a chemical field that is dynamically modified through the collective contributions from the particles. Such systems of reinforced diffusions have been widely studied and their hydrodynamic limits that are nonlinear non-local partial differential equations are usually referred to as Patlak-Keller-Segel (PKS) equations. 
	
	Solutions of the classical PKS equation 
	may blow up in finite time and much of the PDE literature has been focused on understanding this blow-up phenomenon. In this work we study a modified form of the PKS equation that is natural for applications and for which global existence and uniqueness of solutions are easily seen to hold. Our focus here is instead on the study of 
	the long time behavior   through certain interacting particle systems.
	
	Under the so-called ``quasi-stationary hypothesis'' on the chemical field, the limit PDE reduces to a parabolic-elliptic system that is closely related to granular media equations whose time asymptotic properties have been extensively studied probabilistically through certain Lyapunov functions \cite{Mal03, BGV05, CGM06}. The modified PKS equation studied in the current work is a parabolic-parabolic system for which analogous Lyapunov function constructions are not available. A key challenge in the analysis is that the associated interacting particle system is not a Markov process as the interaction term depends on the whole history of the 
	empirical measure.
	
	We establish, under suitable conditions, uniform in time convergence of the empirical measure of particle states to the solution of the PDE. We also provide uniform in time exponential concentration bounds for rate of the above convergence under additional integrability conditions. Finally, we introduce an Euler discretization scheme for the simulation of the interacting particle system and give error bounds that show that the scheme converges uniformly in time and in the size of the particle system as the discretization parameter approaches zero.}

\newenvironment{pf}{\noindent\emph{Proof \,}}{\mbox{}\qed}

\newcommand{\Hess}{\operatorname{Hess}}

\newcommand{\Cstar}{\mathcal{C}^*([0,\infty)\times \R^d)}
\newcommand{\Crd}{\mathcal{C}([0,\infty):\R^{Nd})}
\newcommand{\clf}{\mathcal{F}}
\newcommand{\clc}{\mathcal{C}}
\newcommand{\clp}{\mathcal{P}}
\newcommand{\clw}{\mathcal{W}}
\newcommand{\clh}{\mathcal{H}}
\newcommand{\rr}{r}
\newcommand{\Err}{\mathcal{R}}

\newcommand\R{{\mathbb R}}     

\newcommand\E{{\mathbb E}}     
\newcommand\F{{\mathcal F}}    

\newcommand\eps{\epsilon}
\newcommand\noi{\noindent}
\newcommand\<{{\langle}}

\begin{document}




\section{Introduction}

Consider the following system of nonlinear nonlocal partial differential equations 
\begin{equation}\label{E:PKS_PP}
\left\{\begin{aligned}
\partial_t u(t,x)           &=  \frac{1}{2}\,\Delta u(t,x) -\nabla\cdot \big(u(t,x)[\chi\,\nabla h(t,x)-\nabla V(x)]\big)   \\
\frac{1}{\gamma} \,\partial_t h(t,x) &=  \frac{1}{2}\,\Delta h(t,x)  -\alpha h(t,x) + \beta \int_{\R^d}u(t,z)\,g(z-x)\,dz,   
\end{aligned}\right.
\end{equation}
where $(t,x)\in (0,\infty)\times \R^d $ and $\alpha,\,\beta,\,\gamma,\,\chi$ are positive constants. The symbols $\nabla$, $\nabla \cdot$ and $\Delta$ denote the gradient operator, the divergence operator and the standard Laplacian respectively. 
Equations of the above form arise as reinforced diffusion models for chemotaxis of particles representing biological cells or microorganisms
in which the particle diffusions are directed by the gradient of a chemical field which in turn is dynamically modified by the contributions of the particles themselves(cf. \cite{Pat53, KS70}). 
The  functions $u(t,x)$ and $h(t,x)$ represent,  respectively, the continuum limits of the densities of the biological particles and the particles constituting the chemical field.  The parameters $\alpha$ and $\beta$ in the second equation  model the decay rate of the chemical particles and the rate at which the biological particles contribute to the  chemical field, respectively. The function $g$ is the dispersal kernel which models the  spread and amount of the chemical produced by the biological particles. A natural form for  $g$ is a Gaussian kernel  $g(x,y)\doteq (2\pi \delta)^{-d/2}\exp\{-|y-x|^2/2\delta\}$, where $\delta$ is a small parameter.
The first equation describes the collective motion of the biological particles. The dynamics of the individual particles is coupled through the gradient of the chemical field, i.e. $\nabla h$, which defines their drift coefficient (up to a positive constant multiplier $\chi$). Finally, the function $V$ models a  confinement potential for the particle motions.
Thus the reinforcement mechanism is as follows. Particles are attracted to the chemical and they emit the chemical at a constant rate, resulting in a  positive feedback: the more the cells are aggregated, the more concentrated in their vicinity is the chemical  they produce  which in turn attracts other cells.


The key feature of the model is the competition between the aggregation resulting from the above reinforcement mechanism  and the diffusive effect which spreads out the biological and chemical particles in space. 
When $V=0$ and $g(z-x)dz$ is the Dirac delta measure $\delta_x$, \eqref{E:PKS_PP} becomes the classical Patlak-Keller-Segel (PKS) model \cite{Pat53, KS70} which has been studied extensively.
It is well known that for the $2$-d PKS model (i.e. $d=2$), there is a critical mass $M_c$ such that (i) the solution 
to \eqref{E:PKS_PP} blows up in finite time if the initial mass $\int u(0,x)dx>M_c$, and (ii) a smooth solution exists for all time if $\int u(0,x)dx< M_c$. For $d\ge3$, the blow up of the solution is related to the $L^{d/2}$ norm of the initial density but here the theory is less well developed.
We refer the reader to the survey articles \cite{Hor03, Hor04} for references to the large literature on the PKS model and its variants. 


One line of active research has focused on the prevention of finite time blowup of solutions 
via various modifications of the classical PKS equation that discourage mass concentration (cf. \cite{DR08, BRB11,CHJ12} and references therein.)  The replacement of the Dirac delta measure by a smooth density $g(z-x)dz$, as is considered in the current work, can be regarded as one such natural modification of the PKS model in which chemicals are dispersed by cells over a region of positive area rather than over a single point. It is  easy to see that  there are global unique solutions for general initial conditions for the system \eqref{E:PKS_PP} (cf. Proposition \ref{Wellpose_PP}).
The focus of this work is instead on the study of the long time behavior of \eqref{E:PKS_PP} and their particle approximations, for which very little is known. 
Long time behavior of weakly interacting particle systems of various types has been investigated in many recent works (\cite{Tam, CMV03, Mal03, CGM06, BGV05})
and although our work uses many ideas similar to those in these works, one key distinguishing feature and challenge in the model considered here is that the associated weakly interacting particle system is not a Markov process.  In particular, Lyapunov function constructions that have been extensively used in the proofs in the above works are not available for the model considered here.


Note that if $(u,h)$ solve \eqref{E:PKS_PP} and $\int u(0,x) dx = m$, then $(u/m, h)$ solves \eqref{E:PKS_PP} with $\beta$ replaced by $\beta m$. Thus we can (and will) assume without loss of generality that
$\int u(0,x) dx =1$.
Our starting point is the following probabilistic representation for the solution of \eqref{E:PKS_PP} in terms of a nonlinear diffusion of the McKean-Vlasov type.
\begin{equation}\label{E:NonlinearProcessPP}
\left\{\begin{aligned}
d\bar X_t           &= dB_t - \nabla V(\bar X_t)\,dt+ \chi\,\nabla h(t,\, \bar X_t)\,dt, \\
\frac{1}{\gamma} \,\partial_t h(t,x) &=  \frac{1}{2}\,\Delta h(t,x)  -\alpha h(t,x) + \beta \int_{\R^d}g(z-x)\,d\mu_t(z), \\
\mathcal{L}( \bar X_t) &=\,\mu_t ,
\end{aligned}\right.
\end{equation}
where $\{B_t\}$ is the standard Brownian motion in $\R^d$ and $\mathcal{L}( \bar X_t)$ denotes the probability law of $ \bar X_t$. 
In Proposition \ref{Wellpose_PP} we will show that the above equation has a unique pathwise solution
$(\bar X_t, h(t, \cdot))$ under natural conditions on the initial data and the kernel $g$. 
Furthermore, for $t>0$ the measure $\mu_t$
admits a density $u(t, \cdot)$ with respect to the Lebesgue measure. The first equation  in \eqref{E:PKS_PP} can be regarded as the Kolmogorov's forward equation for the first equation in  \eqref{E:NonlinearProcessPP}. In particular, it is easy to check that the pair $(u(t, \cdot), h(t, \cdot))$ is a solution of \eqref{E:PKS_PP}.  Along with the nonlinear diffusion
\eqref{E:NonlinearProcessPP} we will also study a mesoscopic particle model for the chemotaxis phenomenon described above that is given through a stochastic system of weakly interacting particles of the following form. 
\begin{equation}\label{E:PKS_PP_N}
\left\{\begin{aligned}
dX^{i,N}_t           &=  dB^i_t - \nabla V(X^{i,N}_t)\,dt+ \chi\,\nabla h_N(t,X^{i,N}_t)\,dt,\quad i=1, \dots,\,N   \\
\frac{1}{\gamma}\,\partial_t h_N(t,x) &=  \frac{1}{2}\,\Delta h_N(t,x) -\alpha h_N(t,x) +\frac{\beta}{N} \sum_{i=1}^N g(X^{i,N}_t-x), \; x \in \mathbb{R}^d
\end{aligned}\right. 
\end{equation}
where $\{B^i\}_{i=1}^N$ are independent standard Brownian motions in $\R^d$. Note that the second equation in \eqref{E:PKS_PP_N}  is the same as that in \eqref{E:NonlinearProcessPP} with  $\mu_t$ replaced by the empirical measure 
\begin{equation}\label{def:muN}
\mu^N_t\doteq  \frac{1}{N} \sum_{i=1}^N \delta_{X^{i,N}_t}.
\end{equation}
In this model a detailed evolution for  biological particles is used  whereas the chemical field is regarded as the continuum limit of much smaller chemical molecules. One can also consider a microscopic model where a detailed evolution equation 
of chemical particles replaces the second equation in \eqref{E:PKS_PP_N}.  Such  `full particle system approximations'
of \eqref{E:PKS_PP} for the classical PKS model (i.e. when $g$ is replaced by a Dirac probability measure) were studied in \cite{Ste00} where the convergence of the empirical measures of the biological and chemical particles to the solution of the limit PDE,  up to the blow up time of the solutions, was established.
Starting from the  works of McKean and Vlasov \cite{Mck}, nonlinear diffusions and the associated weakly interacting particle models have been studied by many authors (See, for instance, \cite{Szn91, M96,  Mal03, KK10}.) One important difference in 
\eqref{E:NonlinearProcessPP} (and similarly \eqref{E:PKS_PP_N})
from these classical papers is that the right side of the first equation depends not only on $\mu_t$ but rather on the full past trajectory of the laws, i.e. $\{\mu_s: 0\le s \le t\}$.

The first main goal of this work is to rigorously establish that under suitable conditions, as $N$ becomes large, the mesoscopic model \eqref{E:PKS_PP_N} gives a good approximation for \eqref{E:NonlinearProcessPP} (and thus also for
\eqref{E:PKS_PP}), uniformly in time. Specifically, our results will give, under conditions, uniform in time convergence of  $\mu^N_t$
to $\mu_t$, in a suitable sense. Such a  result is important since it says in particular that the time asymptotic aggregation behavior of the particle system is well captured by the asymptotic density
function $u(t, \cdot)$ as $t\to \infty$. In general one would also like to know how well $\mu^N$ approximates $\mu$ for a fixed value of $N$. In order to address such questions, in our second result, under stronger integrability conditions,  we will provide uniform in time exponential concentration bounds that give estimates on rates of convergence of $\mu^N$ to $\mu$. 

A natural empirical approach for the study of long time properties of \eqref{E:PKS_PP_N} is through numerical simulations. For example, under the Neumann boundary condition, numerical simulations for  \eqref{E:PKS_PE_N} in a square typically demonstrate a separation of time scale: after an initial short time interval during which particles aggregates to form many crowded subpopulations, the subpopulations merge to form a stationary profile at a much slower time scale. See \cite{Fat13} and \cite[Figure 3.9]{SF07} for such simulation results.
Note however that the system cannot be simulated exactly and in practice one needs to do a suitable time discretization. For such simulations to form a reliable basis for mathematical intuition on the long time behavior, it is key that they approximate the system \eqref{E:PKS_PP_N} or the PDE \eqref{E:PKS_PP}, uniformly in time. 
We will show that under suitable conditions a natural discretization scheme for \eqref{E:PKS_PP_N} gives a uniform in time convergent approximation to the solution of \eqref{E:NonlinearProcessPP}  as $N\to\infty$ and as the discretization step size tends to zero. Our uniform in time numerical approximations offer qualitative insights for the long time dynamical behavior of such systems. 



\subsection{Existing results and some challenges}

One of the key challenges in the study of \eqref{E:PKS_PP_N} is that the $Nd$ dimensional process
$X^{(N)}=$ $(X^{1,N}, \cdots , X^{N,N})$ is not a Markov process since the right side of the first equation in \eqref{E:PKS_PP_N}
depends on the full past history of the empirical measure, i.e. $\{\mu^N_s\}_{0\le s \le t}$. In order to get a Markovian descriptor  one needs to consider the pair $(X^{(N)}, h_N)$ which is an infinite dimensional Markov process.
Similar difficulties arise in the study of \eqref{E:PKS_PP} where the form of the  coupling between $u$ and $h$ makes the analysis challenging.

These difficulties do not occur for the reduced parabolic-elliptic system \eqref{E:PKS_PE} obtained by formally letting $\gamma\to\infty$ in \eqref{E:PKS_PP}:
\begin{equation}\label{E:PKS_PE}
\left\{\begin{aligned}
\partial_t u(t,x)           &= \frac{1}{2}\,\Delta u(t,x) -\nabla\cdot \big(u(t,x)[\chi\,\nabla h(t,x)-\nabla V(x)]\big)    \\
0 &=  \frac{1}{2}\,\Delta h(t,x)  -\alpha h(t,x) + \beta \int_{\R^d}u(t,z)\,g(z-x)\,dz.
\end{aligned}\right.
\end{equation}
In the context of chemotaxis, the model in \eqref{E:PKS_PE} corresponds to a quasi-stationary hypothesis for the chemical $h$, that is, the chemical diffuses at a much faster time scale than the biological particles.
Equation \eqref{E:PKS_PE} is mathematically more tractable since here one can solve for $h$ explicitly in terms of $u$ and $g$
as
$$h=\beta G_{\alpha}*u,$$
where $*$ denotes the standard convolution operator, $G_{\alpha}(z)=\int_0^{\infty}e^{-\alpha t}P_tg(z)\,dt$ and $P_t$ is the standard heat semigroup, i.e. the semigroup generated by $\frac{1}{2}\,\Delta$. Using this expression for $h$, the system \eqref{E:PKS_PE} can be expressed as a single equation of the form
\begin{equation}\label{E:VW_PE}
\partial_t u= \frac{1}{2}\,\Delta u +\nabla\cdot(u\,[\nabla V-\chi\,\beta\,\nabla G_{\alpha} \,*\,u]). 
\end{equation}
Kinetic equations of the above form have been well studied in the literature \cite{Mck, Tam, Szn91, CMV03, CGM06, Mal03,BGV05}
where they are sometimes referred to as granular media equations because of their use in the modeling of granular flows
(cf. \cite{BCCP}). An interacting particle approximation for this equation takes the following simple form
\begin{equation}\label{E:PKS_PE_N}
dX^{(N)}_t = dB^{(N)}_t - \nabla\Phi_N(X^{(N)}_t)\,dt, 
\end{equation}
where $X^{(N)}_t = (X^{1,N}_t, \ldots, X^{N,N}_t)$,
$\{B^{(N)}_t\}$ is a standard Brownian motion in $\R^{dN}$ and for ${\bf x}=(x_1,\cdots,x_N)\in\R^{dN}$,
\begin{equation*}
\Phi_N({\bf x})\doteq\sum_{i=1}^N V(x_i)-\frac{\chi\beta}{2N}\sum_{i=1}^N\sum_{j=1}^NG_{\alpha}(x_i-x_j).
\end{equation*}
Note that in this case $X^N$ is a Markov process given as a $Nd$ dimensional diffusion with a gradient form drift.
Law of large number results  and  propagation of chaos properties for such models over a finite time interval that rigorously
connect the asymptotic behavior of  \eqref{E:PKS_PE_N} as $N\to \infty$ with the equation in \eqref{E:VW_PE} are classical and go back to the works of McKean\cite{Mck} and Sznitman\cite{Szn91}. In recent years there has also been significant progress in the study of the time asymptotic behavior of \eqref{E:PKS_PE_N} and \eqref{E:VW_PE}. Under suitable growth and convexity assumptions on $V$ and $G_{\alpha}$, \cite{Tam} studied the existence and local exponential stability of fixed points of \eqref{E:PKS_PE}
by a suitable construction of a Lyapunov function. Similar Lyapunov functions were used in \cite{CMV03, Mal03, CGM06} to establish a uniform in time propagation of chaos property and  convergence of $\mu^N_t(dx)$ to $u(t,x)dx$ along with uniform in time exponential concentration bounds. In the context of the model in \eqref{E:PKS_PE_N} and \eqref{E:VW_PE},  simple modifications of arguments in proofs of \cite[Theorem 1.3]{Mal03} and \cite[Theorem 3.1]{CGM06} imply the following results. 
Suppose 
\begin{equation}\label{E:A}
\<x-y,\,\nabla V(x)-\nabla V(y)\> \geq \lambda |x-y|^2\quad\text{for all }x,\,y\in \R^d
\end{equation} 
and 
\begin{equation}\label{E:lamPE}
\lambda>2\,d \,\beta\chi \frac{\|\Hess g\|_{\infty}}{\alpha}
\end{equation}
where $\|\Hess f\|_{\infty}\doteq\sup_{i,j}\sup_{x}|\partial_{x_i}\partial_{x_j}f(x)|$.
Then as $N\to\infty$,
\begin{equation}
\label{eq:eq821}
\sup_{t\geq 0} \mathcal{W}_2\Big(\mathcal{L}(X^{1,N}_t,\,X^{2,N}_t,\,\cdots,\,X^{k,N}_t),\;\big(u(t,x)dx\big)^{\otimes \,k}\Big) \longrightarrow 0
\end{equation}
for all positive integers $k$, where for $p\ge 1$, $\mathcal{W}_p$ is the Wasserstein-$p$ distance (see Section \ref{sec:cont}) on the space of probability measures on $\mathbb{R}^{dk}$ and $u$ is the solution to \eqref{E:VW_PE}. Under the same assumptions \eqref{E:A} and \eqref{E:lamPE}, a uniform concentration bound of the form 
$$
\sup_{t\geq 0} \mathbb{P} \, ( \mathcal{W}_1(\mu_t^N,\,u(t,x)dx) > \epsilon) \leq
C_1 (1+\epsilon^{-2})\; \exp \left( - C_2 \, N\, \epsilon^2 \right)
$$
is obtained in \cite[Theorem 2.12]{BGV05}, where $C_1,\,C_2\in(0,\infty)$ 
(the condition $\beta+2\gamma>0$ in \cite[Theorem 2.12]{BGV05} is implied by \eqref{E:lamPE}.)
The paper \cite{Mal03} proves
uniform in time weak convergence of empirical measures constructed from an implicit Euler discretization scheme for the Markovian system \eqref{E:PKS_PE_N} to the solution of  \eqref{E:VW_PE}. As remarked earlier,  uniform in time numerical approximations are useful for obtaining qualitative insights for the long time dynamical behavior of such systems.

Much less is known  for the  model \eqref{E:PKS_PP}--\eqref{E:PKS_PP_N}. For the classical parabolic-parabolic Patlak-Keller-Segel PDE  a global existence in the subcritical case (i.e the initial mass is less than $8\pi$) in $\R^2$ is established in \cite{CC08} and the corresponding uniqueness result is established in \cite{CM14}. 
We refer the reader to  references in \cite{CM14} for recent development of the parabolic-parabolic Patlak-Keller-Segel PDE.
None of these works consider particle approximations or long time behavior (however see \cite{CM14} for recent stability results in the plane in a quasi parabolic-elliptic regime.)
The goal of the current work is to develop the theory for the long time behavior of \eqref{E:PKS_PP}--\eqref{E:PKS_PP_N}, analogous to the one for parabolic-elliptic model described above. 
As noted earlier, our approach is inspired by the ideas developed in \cite{Tam, CMV03, Mal03, CGM06, BGV05}.
Our main contributions are as follows.

\subsection{Contributions of this work.}
\label{sec:cont}


In this work we identify conditions under which the  particle system \eqref{E:PKS_PP_N} converges to the  nonlinear process \eqref{E:NonlinearProcessPP} {\em uniformly} over the infinite time horizon and 
construct time stable numerical approximations for \eqref{E:PKS_PP_N} and \eqref{E:NonlinearProcessPP}.
More precisely, the main contributions of this paper are as follows.
\begin{enumerate}
	\item Under suitable conditions,  well-posedness of \eqref{E:NonlinearProcessPP} and \eqref{E:PKS_PP_N} is established in Propositions \ref{Wellpose_Chemotaxis_N} and \ref{Wellpose_PP}. 
	\item  Sufficient conditions for a uniform in time propagation of chaos (POC) property for \eqref{E:PKS_PP_N} are identified in   Theorem \ref{T:UniformPOC_PP}. This implies, under the same conditions, a uniform in time law of large numbers (LLN) for the empirical measures $\mu^N$ (Corollary \ref{cor:UniformPOC_PE_2_1}).
	\item  Under stronger integrability conditions, we establish uniform in time exponential concentration bounds for $\mu^N$ given in Theorem \ref{T:Concen_PP_longtime}. These bounds say that the probability of observing  deviation of $\mu^N_t$ from its LLN prediction $\mu_t$ is exponentially small, uniformly in $t$, as $N$ increases.
	\item An explicit Euler scheme for \eqref{E:PKS_PP_N} is constructed and it is shown that it converges to the solution of \eqref{E:PKS_PP_N} uniformly in time and in $N$(Theorem \ref{EScheme}). 
	Together with the  POC result in 2 this shows that the Euler scheme gives a uniform in time convergent approximation for the nonlinear process as $N\to \infty$ and step size goes to $0$ (Corollary \ref{cor:EScheme}).
\end{enumerate}
Our main condition for uniform in time results in 2,3,4 is Assumption \ref{A:A}. This assumption can be regarded as the analog of condition \eqref{E:lamPE} used in the study of \eqref{E:VW_PE}--\eqref{E:PKS_PE_N}.

The paper is organized as follows.
In Section 2, we present the  basic wellposedness results and introduce our main assumptions. Section \ref{sec:sec3} contains the main results of this work. Finally Section \ref{sec:sec4} is devoted to proofs.

\medskip

\noi {\bf Notation: }
For a Polish space (i.e. a complete separable metric space) $S$, $\mathcal{P}(S)$ denotes the space of all probability measures on $S$. This space is equipped with the topology of weak convergence. 
Distance on a metric space $S$ will be denoted as $d_S(\cdot, \cdot)$ and if $S$ is a  normed linear space $S$ the corresponding norm will be denoted as $\|\cdot\|_S$. If clear from the context $S$ will be suppressed from the notation. The space of continuous functions from an interval $I\subset [0,\infty)$ to $\R^d$ is denoted by $\mathcal{C}(I: \mathbb{R}^{d})$.
The space $\mathcal{C}_T\doteq \mathcal{C}([0,T]:\R^d)$ will be equipped with the usual uniform norm and the Fr\'{e}chet space $\clc \doteq \mathcal{C}([0,\infty):\R^d)$ will be equipped with the distance
$$d_{\clc}(x,y) \doteq \sum_{k=1}^{\infty}\frac{ \|x-y\|_{\clc_k}\wedge 1}{2^k}.$$
Given metric spaces $S_i$, $i= 1, \ldots k$, the distance on the space $S_1\times \cdots \times S_k$ is taken to be the  sum of the $k$ distances:
$$d_{S_1\times\cdots \times S_k}(x,y) \doteq \sum_{i=1}^k d_{S_i}(x_i, y_i), \; x = (x_1, \cdots x_k), y = (y_1, \cdots y_k).$$
The law of a $S$ valued random variable $X$  (an element
of  $\mathcal{P}(S)$), is denoted by $\mathcal{L}(X)$. A collection of $S$ valued random variables $\{X_{\alpha}\}$ is said to be  tight if their laws $\{\mathcal{L}(X_{\alpha})\}$ are tight in $\mathcal{P}(S)$. For a signed measure  
$\mu$ on $S$ and a $\mu$-integrable function $f:S \to \R$, we write
$\int f d\mu$ as $\langle f, \mu\rangle$.
For a polish space $S$, the  Wasserstein-$p$ distance on $\mathcal{P}(S)$ is defined as
\begin{equation}
\label{eq:eq909}
\mathcal{W}_{p}(\mu,\nu)  \doteq \Big(\inf_{\pi}\int\int d_S(x,y)^p\,d\pi(x,y)\Big)^{1/p},\end{equation}
where  the infimum is taken over all probability measures $\pi\in \mathcal{P}(S\times S)$ with marginals $\mu$ and $\nu$.
Let $\mathcal{P}_p(S)$ be the set of $\mu\in\mathcal{P}(S)$ having finite $p$-th moments where $p\in[1,\,\infty)$.
It is well known (cf. \cite[Definition 6.8 and Theorem 6.9]{V08}) that $\mathcal{W}_{p}$ metrizes the weak convergence in $\mathcal{P}_p(S)$.
For $p=1$, the Kantorovich-Rubenstein duality (cf. \cite[Remark 6.5]{V08}) says that 
for probability measures $\mu$ and $\nu$ which have finite first moments,
\begin{equation}\label{E:KR_Duality}
\mathcal{W}_{1}(\mu,\nu) =\sup_{f\in \mbox{Lip}_{1}(S)}\left|\langle f, (\mu - \nu)\rangle \right|,
\end{equation}
where $\mbox{Lip}_{1}(S)$ is the space of Lipschitz functions on $S$ whose Lipschitz constant is at most 1. 

Throughout, $(\Omega,\mathcal{F},\mathbb{P})$ will denote a probability space which is equipped with a
filtration $(\mathcal{F}_{t})$ satisfying the usual conditions.
The symbol $\E$ denotes the expectation with respect to the probability measure $\mathbb{P}$. For a stochastic process $X$ the notation $X_t$ and $X(t)$ will be used interchangeably. 

The space of all bounded continuous functions on $S$ is denoted by $\mathcal{C}_{b}(S)$. The supremum  of a function $f:S\to \mathbb{R}$ is denoted as $||f||_{\infty}\doteq\sup_{x\in S}|f(x)|$. 
Space of  functions with $k$ continuous (resp. continuous and bounded) derivatives will be denoted as $\clc^k(\R^d)$ 
(resp. $\clc^k_b(\R^d)$).
For  $f\in \clc^1_b(\R^d)$ and $g\in \clc^2_b(\R^d)$ we denote
\begin{eqnarray*}
	\|\nabla f\|_\infty=\sup_{x}\Big(\sum_{i=1}^d(\partial_{x_i}f)^2\Big)^{1/2} \quad \text{and}\quad
	\|\Hess g\|_{\infty}= \sup_{i,j}\sup_{x}|\partial_{x_i}\partial_{x_j}g(x)|.
\end{eqnarray*}



\section{Preliminaries and well-posedness}

Note that for a bounded function $g:\R^d\to \R$,  a  solution $h$ to \eqref{E:NonlinearProcessPP}  by the variation of constants formula  satisfies
\begin{equation}\label{E:h_PP}
h(t,x) = Q_{t}\,h_0\,(x) +\gamma\,\beta\int_0^t Q_{t-s}\Big( \int g(y-\cdot)\,d\mu_s(y)\Big)(x)\,ds,
\end{equation}
where $\{Q_t\}_{t\geq 0}$ is the semigroup for $f\mapsto \frac{\gamma}{2} \Delta f-\gamma\alpha f$. That is for suitable $\phi: \R^d \to \R$,
\begin{equation}\label{Def:Q_t}
Q_t\phi(x) \doteq \E[\phi(x+B_{\gamma t})e^{-\gamma\alpha t}],
\end{equation}
where $\{B_t\}$ is a standard $d$-dimensional Brownian motion. 
The equation in \eqref{E:h_PP} can be rewritten as follows.
For $m\in \mathcal{P}\big(\mathcal{C}([0,\infty):\R^d)\big)$, let
\begin{equation}\label{Def:Theta}
\Theta_t^{m}(x)\doteq \int_0^t Q_{t-s}\Big(\int g(y-\cdot)\,dm_s(y)\Big)(x)\,ds,\; x \in \R^d
\end{equation}
where $m_s \in \mathcal{P}(\R^d)$ is the marginal  of $m$ at time $s$, namely $m_s = m \circ (\pi_s)^{-1}$, where $\pi_s: \mathcal{C}([0,\infty):\R^d) \to \R^d$ is the projection map, $\pi_s(w)\doteq w(s)$.
Then \eqref{E:h_PP} is same as
\begin{equation}\label{E:h_PP_2}
h(t,x) = Q_{t}\,h_0\,(x) +\gamma\,\beta\,\Theta_t^{\mu}(x)
\end{equation}
where $\mu\in\mathcal{P}\Big(\mathcal{C}([0,\infty):\,\mathbb{R}^{d})\Big)$ is the probability law of $X=(X_t)_{t\geq 0}$ defined by the first equation in \eqref{E:NonlinearProcessPP}.
Similarly, the solution $h_N$ of  \eqref{E:PKS_PP_N} can be written as
\begin{equation}\label{E:h_PP_N}
h_N(t,x) =  Q_{t}\,h_0\,(x) +\gamma\,\beta\,\Theta_t^{\mu^N}(x),
\end{equation}
where $\mu^N\doteq  \frac{1}{N} \sum_{i=1}^N \delta_{X^{i,N}}$ 
is the corresponding empirical measure.

\subsection{Wellposedness}
This section gives the basic wellposedness results for equations  \eqref{E:NonlinearProcessPP} and \eqref{E:PKS_PP_N} under suitable conditions on the dispersal kernel $g$ and initial chemical field $h_0$. 
Lemma \ref{L:Properties_PPN} below gives a uniform boundedness and a uniform Lipschitz property for  $\nabla h$ and $\nabla h_N$. Its proof is straightforward but is included for completeness in
Section \ref{sec:wellposprf}.
\begin{lemma}\label{L:Properties_PPN}
	Suppose $g,\,h_0\in \clc^2_b(\R^d)$. Define, for $m\in \mathcal{P}\big(\mathcal{C}([0,\infty):\R^d)\big)$ and $(t,x)\in[0,\infty)\times\R^d$,
	\begin{equation}\label{E:h^mu}
	h^{m}(t,x) =  Q_{t}\,h_0\,(x) +\gamma\,\beta\,\Theta_t^{m}(x),
	\end{equation}
	where $\Theta_t^{m}$ is given by \eqref{Def:Theta}.
	Then there exists  $C \in(0,\infty)$ such that
	\begin{eqnarray*}
		\sup_{m}\sup_{t\geq 0}\sup_{x\in\R^d}|\nabla h^{m}(t,x)|& \leq& C \quad \text{and}\\
		\sup_{m}\sup_{t\geq 0} |\nabla h^{m}(t,x)-\nabla h^{m}(t,y)|&\leq& C\,|x-y|,\quad \text{for all } x,\,y\in\R^d,
	\end{eqnarray*}
	where the outside supremum is taken over all $m \in \mathcal{P}\big(\mathcal{C}([0,\infty):\R^d)\big)$.
\end{lemma}
Denote by $\Cstar$ the class of continuous functions $\zeta: [0,\infty)\times \R^d \to \R$ such that
for each $t\ge 0$, $\zeta(t, \cdot)$ is continuously differentiable and $\nabla \zeta(t, \cdot)$ is a bounded Lipschitz function. 
The space $\mathcal{C}^*([0,T]\times \R^d)$ is defined similarly. The above lemma shows that for an arbitrary
$m \in \mathcal{P}\big(\mathcal{C}([0,\infty):\R^d)\big)$,  $h^m \in \Cstar$. In Section \ref{sec:wellposprf}, using Lemma \ref{L:Properties_PPN} we prove the following proposition which gives the  
wellposedness of \eqref{E:PKS_PP_N}.

\begin{proposition}\label{Wellpose_Chemotaxis_N}
	Suppose $g,\,h_0\in \clc^2_b(\R^d)$, $V\in \clc^1(\R^d)$ and $\nabla V$ is  Lipschitz. 
	Let $\mu_0 \in \clp_2(\R^d)$ and $\xi_0^N=(\xi^{1,N}_0, \cdots, \xi^{N,N}_0)$ be a $\clf_0$ measurable square integrable $\R^{dN}$ valued random variable with probability law $\mu_0^{\otimes N}$. Then  the system of equations \eqref{E:PKS_PP_N} has a unique pathwise solution $(X^N,\,h_N) \in 
	\Crd \times \Cstar$  with $(X^N(0),\,h_N(0)) =(\xi_0^N,\,h_0)$.
\end{proposition}
With another application of Lemma \ref{L:Properties_PPN} and  straightforward modifications of  classical fixed point arguments (cf. \cite{Szn91}), we prove the following proposition in Section \ref{sec:wellposprf} as well.
\begin{proposition}\label{Wellpose_PP}
	Suppose $g,\,h_0\in \clc^2_b(\R^d)$, $V\in \clc^1(\R^d)$ and $\nabla V$ is  Lipschitz. Let $\xi_0$ be a $\clf_0$ measurable  $\R^d$ valued random variable with probability law $\mu_0 \in \clp_2(\R^d)$.
	Then  equation \eqref{E:NonlinearProcessPP} has a unique pathwise solution $(X,\,h) \in \mathcal{C}([0,\infty):\R^d) \times \Cstar$ with $(X(0),\,h(0)) =(\xi_0,\,h_0)$.
\end{proposition}

\subsection{Assumptions}

The following will be our standing assumptions.
\begin{itemize} 
	\item
	$\gamma=1$  and  $\int_{\R^d} g(x)\,dx=1$.
	The second assumption can be made without loss of generality by modifying the value of $\beta$ whereas the first assumption is for notational convenience; the proofs for a general $\gamma$ follows similarly.
	
	\item The functions
	$g,\,h_0\in \clc^2_b(\R^d)$.
	
	\item The confinement potential $V \in \clc^1(\R^d)$, is symmetric, $V(0)=0$ and $\nabla V$ is Lipschitz.
	
	\item The initial measure $\mu_0 \in \clp_2(\R^d)$.
\end{itemize}
The above assumptions will be used without further comment. Note that from the last assumption it follows that
$\nabla V(0)=0$.

In addition, for several results the following convexity assumption
will be made. This condition  plays an analogous role in the  study of the long-time properties of the parabolic-parabolic system  as  condition \eqref{E:lamPE} for the parabolic-elliptic system.  
Let
\begin{equation}
v_* \doteq \inf_{x\neq y} \frac{\langle x-y, \nabla V(x) - \nabla V(y)\rangle}{|x-y|^2}.
\label{eq:eq951}
\end{equation}
Note that since $\nabla V$ is Lipschitz, $|v_*| \le L_{\nabla V}$ where latter is the 
the Lipschitz constant of $\nabla V$. 
Let
$$\lambda \doteq \left(\|\Hess h_0\|_\infty + \frac{2\beta \|\Hess g\|_{\infty}}{\alpha}\right)\chi d.$$
\begin{assumption}\label{A:A}
	The confinement potential  $V$ is such that
	\begin{equation}
	v_*>\lambda. \label{E:A1}
	\end{equation}
\end{assumption}

A prototypical example of a $V$ that satisfies Assumption \ref{A:A} is $V(x)=\<x,\,Ax\>/2$ where $A$ is a positive definite $d\times d$ matrix with spectrum bounded from below by $\lambda$. 
\section{Main results}
\label{sec:sec3}
\subsection{Propagation of chaos}\label{subsectionPOC}

A standard approach to proving POC (see, for instance \cite{Szn91, Mal03, CGM06}) is by a coupling method.
Let $\{B^i\}_{i=1}^N$ and $\{\xi_0^{i,N}\}_{i=1}^N$ be  collections of  independent standard $d$-dimensional $\{\clf_t\}$-Brownian motions and $\clf_0$ measurable i.i.d. square integrable random variables with probability law 
$\mu_0 \in \clp_2(\R^d)$, respectively.
Fix  $h_0 \in \clc_b^2(\R^d)$. 
We construct  coupled systems $\{X^{i,N}\}_{i=1}^N$ and $\{\bar{X}^{i}\}_{i=1}^N$ of $d$-dimensional continuous stochastic processes in such a way that
\begin{itemize} 
	\item $\bar{X}^{i}_0=X^{i,N}_0= \xi_0^{i,N}$ for all $i= 1, \cdots N$.
	\item $X^N \doteq (X^{1,N}, \ldots, X^{N,N})$ is the solution to \eqref{E:PKS_PP_N} with driving Brownian motions $\{B^i\}$
	and for each $i=1, \cdots N$, $\bar{X}^i$ is the solution to \eqref{E:NonlinearProcessPP} driven by the  Brownian motion $B^i$. 
\end{itemize}

Using the above coupling we establish the following POC for any finite time horizon. Note this result does not require the convexity assumption (i.e. Assumption \ref{A:A}). 
\begin{theorem}\label{T:POC_PP}
	For each $T\geq 0$, there exists $C_T\in (0,\infty)$ such that
	\begin{equation*}
	\E\Big[\, \sup_{t\in[0,T]}|X^{i,N}_t - \bar{X}_t^i|^2  \,\Big] \le \frac{C_T}{N}.
	\end{equation*}
\end{theorem}
As an immediate consequence of this result we have the following result on asymptotic mutual independence of
$(X^{1,N}, \ldots, X^{k,N})$ for each fixed $k$ and the convergence of each $X^{i,N}$ to $\bar X^1$ (cf. \cite{Szn91}).
The result in particular says that $\mathcal{L}(X^{1,N},\,\cdots,\,X^{N,N})$ is $\mathcal{L}(\bar{X})$-chaotic in the
terminology of \cite{Szn91}.
\begin{corollary}\label{cor:UniformPOC_PE_1_1}
	As $N\to\infty$, we have
	\begin{equation*}
	\mathcal{W}_2\Big(\mathcal{L}(X^{1,N},\,X^{2,N},\,\cdots,\,X^{k,N}),\;\mathcal{L}(\bar{X}^1)^{\otimes \,k}\Big) \longrightarrow 0
	\end{equation*}
	for all  $k \in \mathbb{N}$, where $\mathcal{W}_2$ is the Wasserstein-2 distance on $\mathcal{P}\left(\clc^k\right)$.
\end{corollary}
\begin{pf}
	From the definition of the Wasserstein-2 distance and using the fact that $(X^{i,N}, \bar X^i)$ has same distribution as
	$(X^{1,N}, \bar X^1)$, for $i=1, \ldots, N$, we have
	\begin{eqnarray*}
		&&\mathcal{W}_2\Big(\mathcal{L}(X^{1,N},\,X^{2,N},\,\cdots,\,X^{k,N}),\;\mathcal{L}(\bar{X}^{1},\,\bar{X}^{2},\,\cdots,\,\bar{X}^{k}) \Big) \leq  \sqrt{k}\,\sqrt{\E\, d_{\clc}(X^{1,N}, \bar{X}^{1})^2}.
	\end{eqnarray*}
	The RHS tends to zero as $N\to\infty$ since $\E\|X^{1,N}- \bar{X}^{1}\|^2_{\clc_T}\to 0$ for all $T\geq 0$ by Theorem \ref{T:POC_PP}. The proof is complete since  $\{\bar{X}^{i}\}$ are i.i.d. 
\end{pf}

\vspace{0.1in}
Since $\{X^{i,N}\}_{i=1}^N$ are exchangeable, by \cite[Proposition 2.2]{Szn91}, we have the following  process level weak convergence of  empirical distributions.
\begin{corollary}\label{cor:UniformPOC_PE_1_3}
	As $N\to\infty$, the random measures $\mu^N \doteq \frac{1}{N} \sum_{i=1}^N \delta_{X^{i,N}}$  converge to the deterministic measure $\mathcal{L}(\bar{X})$ in probability in $\mathcal{P}\left(\clc)\right)$.
\end{corollary}
Observe that  Corollary \ref{cor:UniformPOC_PE_1_1} implies in particular that
\begin{equation}\label{E:UniformPOC_PE_1_1}
\sup_{t\in[0,T]} \mathcal{W}_2\Big(\mathcal{L}(X^{1,N}_t,\,X^{2,N}_t,\,\cdots,\,X^{k,N}_t),\;\mathcal{L}(\bar{X}_t)^{\otimes \,k}\Big) \longrightarrow 0
\end{equation}
for all $T\geq 0$.
However this result does not give uniform in time convergence of these multidimensional laws. 
To obtain a uniform in time result, we will make the stronger assumption in Assumption \ref{A:A}. 
The following is the analog of Theorem \ref{T:POC_PP} over an infinite time horizon.
\begin{theorem}\label{T:UniformPOC_PP}
	Suppose Assumption \ref{A:A} is satisfied.
	Then there exists $C\in (0,\infty)$ such that
	\begin{equation*}
	\sup_{t \geq 0 }\E\left[|X^{i,N}_t - \bar{X}_t^i|^2  \right] \le \frac{C}{N}.
	\end{equation*}
\end{theorem}
As an immediate consequence of the theorem we have the following uniform in time propagation of chaos result and uniform in time convergence of the empirical measures $\mu^N(t)$.
\begin{corollary}\label{cor:UniformPOC_PE_2_1}
	Suppose Assumption \ref{A:A} is satisfied. Then for all $N, k\in \mathbb{N}$, we have
	\begin{equation*}
	\sup_{t\geq 0} \mathcal{W}_2\Big(\mathcal{L}(X^{1,N}_t,\,X^{2,N}_t,\,\cdots,\,X^{k,N}_t),\;\mathcal{L}(\bar{X}^1_t)^{\otimes \,k}\Big) \le \frac{C \sqrt{k}}{\sqrt{N}},
	\end{equation*}
	where $\mathcal{W}_2$ is the Wasserstein-2 distance on $\mathcal{P}(\mathbb{R}^{dk})$.  Furthermore, if $\mu_0 \in \clp_q(\R^d)$ for some $q>2$, then
	$$\sup_{t\ge 0} \E\left[\clw^2_2(\mu^N_t, \mu_t)\right] \to 0$$
	as $N\to \infty$, where $\mu^N(t) = \frac{1}{N} \sum_{i=1}^N \delta_{X^{i,N}_t}$ and 
	$\mu_t = \mathcal{L}(\bar X^1_t)$.
\end{corollary}

Proofs of Theorem \ref{T:POC_PP}, Theorem \ref{T:UniformPOC_PP} and Corollary \ref{cor:UniformPOC_PE_2_1}
are given in Section \ref{sec:poc}.


\subsection{Concentration bounds}

In this section, we present our concentration estimates for $\mu^N_t$  in $\mathcal{W}_1$-distance.
As in the previous subsection, we first give a result for finite time horizons (this result will not use Assumption \ref{A:A}). 
\begin{theorem} \label{T:Concen_PP}
	Suppose the initial distribution $\mu_0\in \mathcal{P}(\R^d)$ has a finite square-exponential
	moment, that is,
	there is $\theta_0>0$ such that 
	\begin{equation}\label{eq:eq610}
	\int_{\R^d} e^{\theta_0 \,|x|^2} \, d\mu_0(x) < \infty.
	\end{equation}
	Fix $T\in (0, \infty)$. Then there is a $K\in (0, \infty)$ and,  for any $d'\in (d,\infty)$, $N_0$ and $C$ in $(0, \infty)$, such that
	$$
	\mathbb{P}  \Big(\sup_{0 \leq t \leq T} \mathcal{W}_1(\mu_t^N,\mu_t) > \epsilon\Big) \leq
	C (1+\epsilon^{-2}) \exp \left( - K \, N\, \epsilon^2 \right)
	$$
	for all $N \geq N_0\,\max(\epsilon^{-(d'+2)},1)$ and $\epsilon>0$.
\end{theorem}
We note that the constant $K$ may depend on $T$ but not on $\eps$ and $d'$; also $C$ and $N_0$ may depend on $T$ and $d'$
but not on $\eps$.
The main idea  in the proof is, as in \cite{BGV05}, to (i)  bound $\mathcal{W}_1(\mu_t^N,\mu_t)$ in terms of $\big(\mathcal{W}_1(\nu_s^N,\mu_s)\big)_{s\in[0,t]}$ where
\begin{equation}\label{def:munuN}
\nu^N_s\doteq  \frac{1}{N} \sum_{i=1}^N \delta_{\bar{X}^{i}_s},
\end{equation}
and $\{\bar{X}^i\}_{i=1}^N$ are the processes defined at the beginning of Section \ref{subsectionPOC}, and
then; (ii) estimate $\big(\mathcal{W}_1(\nu_s^N,\mu_s)\big)_{s\in[0,t]}$ which is a quantity that concerns i.i.d. random variables $\{\bar{X}^i\}_{i=1}^N$. The first step is accomplished in Subsection \ref{s:4.3.1} via a coupling argument similar to the one used in the proof of results in Section \ref{subsectionPOC}, while the second step relies on 
an estimate from \cite{BGV05} for the tail probabilities for empirical measures of i.i.d. random variables that
is based on the equivalence between Talagrand's transportation inequalities (cf. \cite{BGV05}) and existence of a finite
square-exponential moment. The precise result obtained in \cite{BGV05} is as follows.
For $a,\alpha\in (0,\infty)$ we let
$$\clp_{a,\alpha} \doteq \{\nu \in \clp(\R^d): 	\int_{\R^d} e^{\alpha \,|x|^2} \, d\nu(x) \le a\}.$$
\begin{theorem}\cite[Theorem 2.1]{BGV05}\label{T:BGV1.1}
	Fix $a,\alpha\in (0,\infty)$.	Then, there is a $\theta > 0$ such that for any $d'\in (d,\infty)$  there exists a positive integer $N_0$ such that
	\begin{equation*}
	\sup_{\nu \in \clp_{a,\alpha}}
	\mathbb{P}\left( \mathcal{W}_1(\hat{\nu}^N,\,\nu)>\eps\right)\le e^{- \, \frac{\theta}{2} \, N \, \eps^2}.
	\end{equation*}
	for all $\eps >0$ and $N \geq N_0\,\max(\eps^{-(d'+2)},\,1) $, where $\hat{\nu}^N \doteq \frac{1}{N} \sum_{i=1}^N \delta_{Z^i}$ and $(Z^i)_{i \in \mathbb{N}}$ are iid random variables with  law $\nu$.
\end{theorem}

We shall apply this theorem to $\nu=\mu_s$. In order to do so, we need  $\mu_s$ to have a finite squared-exponential moment. 
We will show in Section \ref{s:4.3.2} that if $\mu_0$ satisfies \eqref{eq:eq610}, then for every $T>0$ there is a
$\theta_T \in (0, \theta_0)$ such that  
\begin{equation}\label{eq:eq610T}
\sup_{s\in [0,T]}\int_{\R^d} e^{\theta_T \,|x|^2} \, d\mu_s(x) < \infty.
\end{equation}
This will allow us to apply Theorem \ref{T:BGV1.1} in completing step (ii) in the proof
of Theorem \ref{T:Concen_PP}.

We  next show in Theorem \ref{T:Concen_PP_longtime} below that when the convexity property in Assumption \ref{A:A} is satisfied then a {\em uniform in time} concentration bound holds. The key step (see Section \ref{s:4.3.2}) is to argue (see Proposition \ref{prop:Smoment}) that under this assumption, for some $\theta_{\infty} >0$, \eqref{eq:eq610T}
holds with $[0,T]$ and $\theta_T$  replaced with $[0,\infty)$ and $\theta_{\infty}$ respectively.
This together with another uniform bound established in Section \ref{s:4.3.1} (Proposition \ref{prop:W1PP}) will imply the uniform in time  concentration bound given in the theorem below. 
\begin{theorem}\label{T:Concen_PP_longtime}
	Suppose that $\mu_0$ satisfies \eqref{eq:eq610} for some $\theta_0>0$. Suppose further that Assumption \ref{A:A} is satisfied.
	Then there exists $K \in (0, \infty)$ such that
	for any $d'\in (d,\infty)$, there exist $C\in (0, \infty)$ and $N_0\in (0, \infty)$ such that
	$$
	\sup_{t\geq 0} \mathbb{P} \, ( \mathcal{W}_1(\mu_t^N,\mu_t) > \epsilon) \leq
	C (1+\epsilon^{-2})\; \exp \left( - K \, N\, \epsilon^2 \right)
	$$
	for all $N \geq N_0 \max(\epsilon^{-(d'+2)},1) $ and $\epsilon>0$.
\end{theorem}
We note that $C$ and $N_0$ may depend on $d'$ but not on $\eps$.\\ 

Proofs of Theorems \ref{T:Concen_PP} and \ref{T:Concen_PP_longtime} will be given in Section \ref{sec:sec-concbd}.
\subsection{Uniform convergence of Euler scheme}
In this section  we will introduce an Euler approximation for the collection of SDE in \eqref{E:PKS_PP_N}
which can be used for approximate simulation of the system.  We show that the approximation error converges to $0$ as the time discretization size $\eps$ converges to $0$, {\em uniformly in time}. As a consequence it will follow that
the empirical measure of the particle states in the approximate system converges to the law of the nonlinear process, uniformly in time, as $N\to \infty$ and $\eps \to 0$ (Corollary \ref{cor:EScheme}).


Note that $Q_t$ has transition density $q(t,x,y)=e^{-\alpha t}p( t,x,y)$ with respect to Lebesgue measure, where $p(t,x,y)$ is the standard Gaussian kernel. 
Using \eqref{E:h_PP_N}, the system of equations governing the particle system $X^{(N)} = (X^{i,N})_{i=1}^N$ in \eqref{E:PKS_PP_N}  can be written as
\begin{equation}\label{PPN2}
dX_t^{i,N} = dB_t^i + \Big(\int_0^t G_{t-s}(X_s^{(N)},X_t^{i,N})\,ds - \nabla V_t(X_t^{i,N})
\Big)\,dt 
\end{equation}
for $1\leq i\leq N$, where 
\begin{equation}
\label{eq:eq226}
V_t = V-\chi Q_t h_0
\end{equation}
and for $\vec{x} = (x_1,\dots,x_N)\in \R^{dN}$ and $y\in \R^d$, 
\begin{align}\label{G_theta}
G_\theta(\vec{x},y) 
&= \frac{\chi\beta}{N}\sum_{i=1}^N\int_{\R^d} \nabla_y q(\theta,y,z)g(x_i-z)\,dz.
\end{align}
We now define an explicit Euler scheme for  \eqref{PPN2} with step size $\epsilon \in (0,1)$. Let $Y_0^{(N),\epsilon} = X_0^{(N)}$. Having defined $Y_n^{(N),\epsilon} = (Y_n^{i,N,\epsilon})_{i=1}^{N}$ for some $n\ge 0$,  we define $Y_{n+1}^{(N),\epsilon}$ naturally as
\begin{equation} \label{EPPN2}
Y_{n+1}^{i,N,\epsilon} \doteq Y_n^{i,N,\epsilon}+\Delta_nB^i + \epsilon\,\Big( \int_{0}^{n\epsilon}G_{n\epsilon-s}(\tilde{Y}_s^{(N),\epsilon},\,Y_n^{i,N,\epsilon})\,ds - \nabla V_{n\epsilon}(Y_n^{i,N,\epsilon})\Big)
\end{equation}
for $1\leq i\leq N$,
where $\Delta_n B^i \doteq B_{(n+1)\epsilon}^i-B_{n\epsilon}^i$  and
\begin{equation*}
\tilde{Y}_s^{(N),\epsilon} \doteq Y_k^{(N),\epsilon} \quad\text{for }s\in[k\epsilon,(k+1)\epsilon).
\end{equation*}  
Note that the integral on the right hand side of  \eqref{EPPN2} can be written as
$$\int_{0}^{n\epsilon}G_{n\epsilon-s}(\tilde{Y}_s^{(N),\epsilon},\,Y_n^{i,N,\epsilon})\,ds= \sum_{k=0}^{n-1} \int_{k\epsilon}^{(k+1)\epsilon}G_{n\epsilon-s}(Y_k^{(N),\epsilon},\,Y_n^{i,N,\epsilon})\,ds.$$
Thus in order to evaluate a typical Euler step, one needs to compute terms of the form
$\int_{[k\eps, (k+1)\eps]} G_\theta(\vec{x},y) d\theta$ which can be done using numerical integration.

Our goal is to provide uniform in time estimates on the mean square error of the scheme, namely to estimate the quantity 
$$\E\big| Y_n^{i,N,\epsilon}-X_{n\epsilon}^{i,N}\big|^2.$$
For that we begin by establishing moment bounds for the Euler scheme which are uniform in $N$, step size $\epsilon$ and  time instant $n$. 
Recall $v_*$ introduced in \eqref{eq:eq951}. Also recall that $\mu_0 \in \clp_2(\R^d)$.

\begin{lemma}\label{L:Ybound}
	Suppose $v_*>0$. Then there exists $\epsilon_0\in (0,1)$ such that
	\begin{equation*}
	\sup_{\epsilon\in (0,\,\epsilon_0)}\sup_{N\in \mathbb{N}_0} \sup_{1\le i\le N} \sup_{n\ge0} \E|Y_n^{i,N,\epsilon}|^2<\infty.
	\end{equation*}
\end{lemma}
We now present our main result on the uniform convergence of the Euler scheme.
For this result we will make the stronger convexity assumption in Assumption \ref{A:A}. 

\begin{theorem}\label{EScheme}
	Suppose Assumption \ref{A:A} holds.
	Then there exists $\epsilon_0 \in(0,1)$ and $C\in(0,\infty)$ such that 
	\begin{equation}\label{E:EScheme}
	\sup_{N\ge1}\sup_{1\le i\le N}\sup_{n\in \mathbb{N}_0} \E\big| Y_n^{i,N,\epsilon}-X_{n\epsilon}^{i,N}\big|^2 \le C\,\epsilon
	\end{equation}
	for all $\epsilon\in(0,\,\epsilon_0)$.
\end{theorem}

It is important that the estimate in \eqref{E:EScheme} is uniform not only in time instant $n$ but also in the size of the system $N$. As a consequence one has the desired property that in order to control the mean square error for larger systems one does not need  smaller time discretization steps.  This in particular implies that the Euler scheme provides a good numerical approximation to the nonlinear process, uniformly in time. Namely we have the following result.
\begin{corollary}\label{cor:EScheme}
	Suppose Assumption \ref{A:A} holds. There exists $\eps_0 \in (0,1)$ and $C\in(0,\infty)$ such that for any positive integer $k$,
	\begin{equation*}
	\sup_{n\geq 0} \mathcal{W}_2\Big(\mathcal{L}(Y^{1,N,\epsilon}_n,\,Y^{2,N,\epsilon}_n,\,\cdots,\,Y^{k,N,\epsilon}_n),\;\mathcal{L}(\bar{X}_{n\epsilon})^{\otimes \,k}\Big) \leq C\sqrt{k}\,\Big(\sqrt{\epsilon}+\frac{1}{\sqrt{N}}\Big)
	\end{equation*}
	for all $\eps \in (0,\eps_0)$. Furthermore if $\mu_0 \in \clp_q(\R^d)$ for some $q\in(2,\infty)$.
	Then
	we have
	\begin{equation}\label{EmpiY}
	\limsup_{N\to\infty}\,\sup_{n\geq 1}\E[\mathcal{W}^2_2(\mu_n^{N,\eps}, \mu_{n\eps})]\leq 2\,C\epsilon
	\end{equation}
	for all  $\epsilon\in(0,\,\epsilon_0)$, where $\mu_n^{N,\eps} \doteq \frac{1}{N} \sum_{i=1}^N \delta_{Y^{i,N,\epsilon}_n}$.
\end{corollary}

Proofs of Lemma \ref{L:Ybound}, Theorem \ref{EScheme} and Corollary \ref{cor:EScheme} will be given in Section \ref{sec:secpfeuler}.

\bigskip
\section{Proofs}
\label{sec:sec4}
We will denote by $\kappa,\,\kappa_1,\,\kappa_2,\cdots$ the constants that appear in various estimates within
a proof. These constants only depend on the  model parameters or problem data, namely, $\alpha,\,\beta,\,\chi,\,g,\,V,\,h_0,d$ and $\mu_0$.
For estimates on a finite time horizon $[0,T]$, these constants may also depend on $T$ and in that case we use $\kappa_T,\,\kappa_{1,T},\,\kappa_{2,T},\cdots$ to denote such constants.
The value of such constants may change from one proof to another.

\subsection{Wellposedness}
\label{sec:wellposprf}

\begin{proof}[Proof of Lemma \ref{L:Properties_PPN}]
	From the definition of $h^m$ in \eqref{E:h^mu} and of the semigroup $\{Q_t\}$, we have for all $x, x_1, x_2 \in \R^d$ and $t \ge 0$,  uniform bounds
	\begin{align*}
	|h^m(t,x)|&\leq e^{-\alpha t}|P_{ t}h_0(x)| +\frac{\beta(1-e^{-\alpha t}) \|g\|_{\infty}}{\alpha}\\
	&\leq e^{-\alpha t}\|h_0\|_{\infty} + \frac{\beta\, \|g\|_{\infty}}{\alpha},\\
	|\nabla h^m(t,x)| &\leq e^{-\alpha t}\|\nabla h_0\|_{\infty} + \frac{\beta\, \|\nabla g\|_{\infty}}{\alpha},\\
	|\nabla h^m(t,x_1)-\nabla h^m(t,x_2)|
	&\leq |x_1-x_2|\,d\,\Big(  e^{-\alpha t}\|\Hess h_0\|_{\infty}+ \frac{\beta \,\|\Hess g\|_{\infty}}{\alpha}\Big).
	\end{align*}
	The result is immediate from the above inequalities.
\end{proof}

\begin{proof}[Proof of Proposition \ref{Wellpose_Chemotaxis_N}]
	For notational simplicity, we  suppress the index $N$ and write $X^{N,i}$ as $X_i$.
	
	\textbf{Uniqueness. }Suppose  $(X,\, h)$ and $(\tilde{X},\, \tilde{h})$ are two solutions to  \eqref{E:PKS_PP_N} with $h(0,\cdot)=\tilde{h}(0,\cdot)=h_0$, 
	and $X(0) = \tilde X(0)$,
	where $X=(X_1,\cdots,\,X_N)$ and $\tilde{X}=(\tilde{X}_1,\cdots,\,\tilde{X}_N)$. Letting $Y_i \doteq X_i-\tilde{X}_i$ and $H\doteq h-\tilde{h}$, we have
	\begin{eqnarray}
	Y_i(t) &=& \int^t_0 \left(-  \nabla V(X_i(s))+ \nabla V(\tilde{X}_i(s)) + \chi\nabla h(s,X_i(s))-\chi\nabla\tilde{h}(s,\tilde{X}_i(s))\right)\,ds, \label{E:Y_it}\\
	H(t,x) &=& \frac{\beta}{N}\sum^N_{i=1}\int^t_0\left(Q_{t-s}\big(g(X_i(s)-\cdot)-g(\tilde{X}_i(s)-\cdot)\big)(x)\right) ds.\label{E:H(t,x)}
	\end{eqnarray}
	
	From  \eqref{E:Y_it} we have for all $t\ge 0$
	\begin{align}
	\sup_{0\leq s\leq t}|Y_i(s)| 
	&\leq  d\int^t_0 \left(\|\Hess V\|_{\infty}\,\big|Y_i(s)\big|+\big|\nabla H(s, \tilde X_i(s))\big| + \chi \big|\nabla h(s, X_i(s))-\nabla h(s, \tilde{X}_i(s))\big|\right)\,ds \notag\\
	&\leq \kappa_1 \int^t_0 \left(\big|Y_i(s)\big|+  \|\nabla H(s)\|_{\infty}\right)\,ds,\label{E:1}
	\end{align}
	where $\|\nabla H(s)\|_{\infty} \doteq \sup_x|\nabla H(s,x)|$ and the last inequality follows 
	from Lemma \ref{L:Properties_PPN} on noting that $h$ equals $h^{\mu^N}$, where $\mu^N \doteq \frac{1}{N}\sum_{i=1}^N \delta_{X_i}$
	is the path empirical measure (see \eqref{E:h_PP_N}).
	From \eqref{E:H(t,x)}, the fact that $\nabla_xp(t,x,y)=-\nabla_yp(t,x,y)$ and integration by parts, we obtain
	\begin{eqnarray*}
		\nabla H(t,x) = \frac{\beta}{N}\sum^N_{i=1}\int^t_0 e^{-\alpha(t-s)}\int_{\R^d}p((t-s),x,y) \nabla_y\big(g(X_i(s)-y)-g(\tilde{X}_i(s)-y)\big)\,dy\,ds.
	\end{eqnarray*}
	Hence since $g \in \clc_b^2(\R^d)$,
	\begin{equation}\label{E:2}
	\|\nabla H(t)\|_{\infty}
	\leq \kappa_2\int^t_0 e^{-\alpha(t-s)} \max_{1\leq i\leq N}\big|Y_i(s)\big| ds.
	\end{equation}
	
	Combining \eqref{E:1} and \eqref{E:2}, and letting $Y(t)\doteq \max_{1\leq i\leq N}\big|Y_i(t)\big| $, we obtain
	\begin{align}
	\sup_{0\leq s\leq t}Y(s)&\leq  \kappa_3 \int^t_0 \left(Y(s) + \Big( \int^s_0e^{-\alpha(s-r)}Y(r)\,dr\Big) \right) ds\nonumber \\
	&\le \kappa_{4,t} \int^t_0  \sup_{0\le r \le s} Y(r) ds . \label{E:Y(s)}
	\end{align}
	This implies $Y(t)=0$ for all $t\geq0$. Finally, from \eqref{E:H(t,x)} we have $H(t,x)=0$ for all $t\ge 0$ and $x \in \R^d$. This completes the proof of pathwise uniqueness.\\
	
	\textbf{Existence. } This is argued by a minor modification of the standard Picard approximation method as follows. Define a sequence $\{(X^{(k)},\,h^{(k)})\}_{k\geq 1}$, where $X^{(k)}=(X^{(k)}_1,\cdots,\,X^{(k)}_N)$,
	of $\Crd \times \Cstar$ valued random variables as follows. Let $X^{(1)}(t)=(\xi_0^{1,N},\cdots,\xi_0^{N,N})$ and $h^{(1)}(t,x)=h_0(x)$ for all $t$. We then define, for $k\geq 2$,
	\begin{eqnarray*}
		X^{(k+1)}_i(t) &\doteq &  \xi_0^{i,N} + \int_0^t\left(-\nabla V(X^{(k)}_i(s)) + \nabla h^{(k)}(s, X^{(k)}_i(s))\right)\, ds + B^i_t,\quad i=1,\dots,\,N,\\
		h^{(k+1)}(t,x) &\doteq & Q_{t}\,h_0\,(x) + \frac{\beta}{N}\sum^N_{i=1}\int^t_0Q_{t-s}\big(g(X^{(k+1)}_i(s)-\cdot)\big)(x)\,ds,
	\end{eqnarray*}
	%
	Let $Y^{(k)}(t)\doteq \max_{1\leq i\leq N}\big|X^{(k+1)}_i(t)-X^{(k)}_i(t)\big|$.
	By similar estimates as that were used to obtain \eqref{E:Y(s)}, we have
	\begin{equation*}
	\sup_{0\leq s\leq t}Y^{(k)}(s)\leq  \kappa_5 \int^t_0 \left(Y^{(k-1)}(s) + \Big( \int^s_0e^{-\alpha(s-r)}Y^{(k-1)}(r)\,dr\Big)\right) \,ds.
	\end{equation*}
	Hence, for fixed $T>0$ and $t \in[0,T]$,
	\begin{equation*}
	\Big(\sup_{0\leq s\leq t}Y^{(k)}(s)\Big)^2\leq  \kappa_{6,T} \int^t_0 (Y^{(k-1)}(s))^2  \,ds .
	\end{equation*}
	A standard iteration argument  then yields
	\begin{equation*}
	\E\Big[\Big(\sup_{0\leq s\leq t}Y^{(k+1)}(s)\Big)^2\Big]\leq  C_T\,\frac{(\kappa_{7,T})^k}{k!};\quad 0\leq t\leq T,\,k\geq 1,
	\end{equation*}
	where $C_T=\max_{1\leq i\leq N}\sup_{0\leq t\leq T}\E[|X^{(2)}_i(t)-X^{(1)}_i(t)|^2]$ is finite from the  uniform boundedness of $\nabla h_N$ proved in Lemma \ref{L:Properties_PPN} and the Lipschitz property
	of $\nabla V$. From this  we conclude that $\{X^{(k)}(t)\}_{0\le t \le T}$ converges a.s in $\clc([0,T]:\R^{Nd})$ to a continuous process
	$\{X(t)\}_{0\le t \le T}$.
	On other hand, using estimates similar to those used in obtaining \eqref{E:2} we have 
	\begin{equation*}
	\|h^{(k+1)}(t)-h^{(k)}(t)\|_{\infty} + \|\nabla h^{(k+1)}(t)-\nabla h^{(k)}(t)\|_{\infty}\leq \kappa_8 \int^t_0 e^{-\alpha(t-s)}\sup_{0\le r \le s}Y^{(k)}(r)  \,ds.
	\end{equation*}
	From this it follows that for every $t \in [0,T]$, $h^{(k)}(t, \cdot)$ converges uniformly to a continuously differentiable function $h(t, \cdot)$  and
	$\nabla h^{(k)}(t,\cdot)$ converges uniformly to $\nabla h(t, \cdot)$. Furthermore the convergence is uniform in $t \in [0,T]$, namely
	$$
	\sup_{0\le t \le T} \left(\|h^{(k)}(t)-h(t)\|_{\infty} + \|\nabla h^{(k)}(t)-\nabla h(t)\|_{\infty}\right) \to 0$$
	as $k \to \infty$.
	It is easy to verify that $h \in \Cstar$ and
	\begin{align*}
	X_i(t) &= \xi_0^{i,N} + \int_0^t\left(-\nabla V(X_i(s)) + \nabla h(s, X_i(s))\right)\, ds + B^i_t,\quad i=1,\dots,\,N,\\
	h(t,x) &=  Q_{t}\,h_0\,(x) + \frac{\beta}{N}\sum^N_{i=1}\int^t_0Q_{t-s}\big(g(X_i(s)-\cdot)\big)(x)\,ds.
	\end{align*}
	This establishes the desired existence of solutions.
\end{proof}

\begin{proof}[Proof of Proposition \ref{Wellpose_PP}]
	The proof uses classical arguments from  \cite{Szn91}.
	It suffices to show that for each $T>0$, equation \eqref{E:NonlinearProcessPP}
	has a unique solution over the time horizon $[0,T]$ which belongs to 
	$\mathcal{C}([0,T]:\R^d)\times \mathcal{C}^*([0,T]\times\R^d)$. Let $T>0$ be arbitrary.
	We note that a probability measure $m \in \mathcal{P}(\mathcal{C}^T)$ can be mapped naturally to a $\hat m \in \mathcal{P}(\mathcal{C}([0,\infty):\mathbb{R}^{d}))$ as $\hat m \doteq m \circ [\pi^T]^{-1}$ where
	$\pi^T: \mathcal{C}([0,T]:\mathbb{R}^{d}) \to \mathcal{C}([0,\infty):\mathbb{R}^{d})$ is defined
	as $(\pi^Tw)(s) \doteq w(s\wedge T)$ for $s \ge 0$, $w \in \mathcal{C}([0,T]:\mathbb{R}^{d})$. Abusing notation
	we denote for $t \in [0,T]$ $\Theta^{\hat m}_t$ as $\Theta^{m}_t$ where $\Theta_t^{m}$ is defined in \eqref{Def:Theta}.

	Define $\Phi:\, \mathcal{P}(\mathcal{C}^T) \to \mathcal{P}(\mathcal{C}^T)$ which maps $m$ to the law $\mathcal{L}(Z)$, where $Z=(Z_t)_{t\in[0,T]}$ is the solution of
	\begin{equation}\label{E:PP_Zm}
	Z_t = \xi_0 + B_t + \int^t_0\left[- \nabla V(Z_s) +  \chi\,\nabla h^m(s,Z_s)\right]\,ds,\quad t\in [0,T],
	\end{equation}
	and $h^m$ is as in \eqref{E:h^mu}. From Lemma \ref{L:Properties_PPN} $\nabla h^m$ is a Lipschitz map and by assumption $\nabla V$ is Lipschitz as well, thus the equation in \eqref{E:PP_Zm} has a unique solution and consequently
	the function $\Phi$ is well-defined.
	Observe that $(X,\,h) $ is a solution of \eqref{E:NonlinearProcessPP} over $[0,T]$ if and only if  $\mathcal{L}(X)\in \mathcal{P}(\mathcal{C}^T)$ is a fixed point of $\Phi$ and $h$ is given by the right hand side of \eqref{E:h_PP} with $\mu_t$ being the law of $X_t$.
	We will show that for all $m^1,\,m^2 \in \mathcal{P}(\mathcal{C}^T)$, we have
	\begin{equation}\label{T:Contraction_McV}
	D_t(\Phi(m^1),\Phi(m^2)) \leq \kappa_T\,\int^t_0D_s(m^1,m^2)\,ds,\quad t\in[0,T],
	\end{equation}
	for some  $\kappa_T\in(0,\infty)$ where  $D_t$ is the  Wasserstein-1 distance on $\mathcal{P}(\mathcal{C}^t)$, namely, for
	$m^1, m^2 \in \mathcal{P}(\mathcal{C}^t)$, $D_t(m^1,m^2)$ is given by the right side of \eqref{eq:eq909} with
	$p=1$ and $S = \mathcal{C}^t$.
	Suppose for $i=1,2$, $Z^i$ solves \eqref{E:PP_Zm} with $m$ replaced by $m^i$ on the right side where $m^i \in \mathcal{P}(\mathcal{C}^T)$. Then, $Z^i$ has law $\Phi(m^i)$ and
	\begin{eqnarray}\label{E:Wellpose_MKV}
	\sup_{s\leq t}|Z^1_s-Z^2_s|& \leq& \int^t_0\Big[|\nabla V(Z^2_s)- \nabla V(Z^1_s)| +  \,\chi\,|\nabla Q_sh_0(Z^1_s)- \nabla Q_s h_0(Z^2_s)| \notag\\
	&&\qquad  +\beta \,\chi\,|\nabla\Theta^{m^1}_s(Z^1_s)-\nabla\Theta^{m^2}_s(Z^2_s)|\Big]\;ds
	\end{eqnarray}
	From properties of the heat semigroup  $|\nabla Q_s h_0(x)-\nabla Q_s h_0(y)| \leq e^{-\alpha s}|x-y| d \|\Hess h_0\|_{\infty}$ and
	\begin{align*}
	& |\nabla\Theta^{m^1}_t(x)-\nabla\Theta^{m^2}_t(y)|\\
	&\quad \leq  \Big|\int^t_0\int_{\R^d} q(t-s,x,z)
	\<\nabla_zg(\cdot-z), m^1_s\>-q(t-s,y,z)
	\<\nabla_zg(\cdot-z),m^2_s\>\,dz\,ds \,\Big|\\
	& \quad \leq \Big|\int^t_0\int_{\R^d} q(t-s,x,z)\<\nabla_zg(\cdot-z),m^1_s\>
	-q(t-s,y,z)\<\nabla_zg(\cdot-z),m^1_s\>\,dz\,ds \,\Big|\\
	& \quad\quad+\int^t_0\int_{\R^d}q(t-s,y,z)\,\big|\<\nabla_zg(\cdot-z),\;m^1_s-m^2_s\>\big|\,dz\;ds\\
	&\quad\leq  d\|\Hess g\|_{\infty}|x-y|\int^t_0e^{-\alpha(t-s)} \,ds\\
	&\quad\quad+\int^t_0\int_{\R^d}q(t-s,y,z)\,\big|\<\nabla_zg(\cdot-z),\; m^1_s-m^2_s\>\big|\,dz\;ds\\
	&\quad\leq \kappa\,\bigg(
	|x-y|+
	\int^t_0 e^{-\alpha(t-s)}\int_{\R^d\times\R^d}|w_1(s)- w_2(s)|\,dM(w_1,w_2)\;ds\bigg),
	\end{align*}
	for any $M \in \mathcal{P}(\mathcal{C}^t\times \mathcal{C}^t)$ with marginals  $m^1$ and $m^2$,
	where the last step uses the fact that $y \mapsto \nabla_zg(y-z)$ is Lipschitz.
	Combining the above estimates with \eqref{E:Wellpose_MKV} and using the Lipschitz property of
	$\nabla V$, we obtain
	\begin{equation*}
	\sup_{s\leq t}|Z^1_s-Z^2_s|
	\leq \kappa_1\int^t_0 \left[|Z^1_s-Z^2_s|+\, \int_{\R^d\times\R^d}\sup_{r\leq s}|w_1(r)- w_2(r)|\,dM(w_1,w_2)\right]\;ds
	\end{equation*}
	for any $M$ as above. Hence
	\begin{equation*}
	\sup_{s\leq t}|Z^1_s-Z^2_s| \leq \kappa_2\int^t_0\left[|Z^1_s-Z^2_s|+ D_s(m^1,m^2)\right]\;ds.
	\end{equation*}
	By Gronwall's Lemma, it now follows that
	\begin{equation*}
	\sup_{s\leq t}|Z^1_s-Z^2_s|\leq \kappa_{3,T}\,\int^t_0D_s(m^1,m^2)ds,\quad t\leq T,
	\end{equation*}
	Taking expectations we obtain \eqref{T:Contraction_McV}.	
	Now by a standard fixed point argument, there exists a unique $m^{*}\in \mathcal{P}(\mathcal{C}^T)$ such that $m^{*}=\Phi(m^{*})$. Let $Z^{*}$ be the unique solution to \eqref{E:PP_Zm} with $m$ replaced by $m^{*}$. Then \eqref{E:NonlinearProcessPP} has a unique pathwise solution $(Z^{*},\,h^{*})$ where  $h^{*}$ is given by the right hand side of \eqref{E:h_PP} with $\mu_t(dy)$ being the law of $Z^{*}_t$.
\end{proof}


\subsection{Propagation of chaos}
\label{sec:poc}
The proofs of Theorems \ref{T:POC_PP} and \ref{T:UniformPOC_PP} are based on breaking the pathwise deviation  
$$\Delta^i_s\doteq X^{i,N}_s - \bar{X}_s^i$$ 
into several manageable terms.
For any $N$ and $i$, integration by parts  yields
\begin{align*}
\,d|\Delta^i_t|^2
=&\, 2 \,\Delta^i_t\cdot \Big[ - \nabla V(X^{i,N}_t) + \nabla V (\bar{X}_t^i) +
\chi \left(\nabla h^{\mu^N}(t, X^{i,N}_t) - \nabla h^{\mu}(t, \bar X^{i}_t)\right)\Big]\, dt
\end{align*}
Note that, with $v_*$ as in \eqref{eq:eq951}
$$\Delta^i_t\cdot \left[- \nabla V(X^{i,N}_t) + \nabla V (\bar{X}_t^i)\right] \le -v_* |\Delta^i_t|^2.$$
Next, from \eqref{E:h^mu}
\begin{align}
&\nabla h^{\mu^N}(t, X^{i,N}_t) - \nabla h^{\mu}(t, \bar X^{i}_t)\nonumber \\
&\quad= \left( \nabla Q_th_0(X^{i,N}_t) - \nabla Q_th_0(\bar X^{i}_t) \right)\nonumber\\
&\quad \quad+ \beta \, \int_0^t\left[
\nabla Q_{t-s}\left( \int g(y-\cdot\,)\,\mu^N_s(dy)\right)(X^{i,N}_t)
- \nabla Q_{t-s}\left( \int g(y-\cdot\,)\,\mu_s(dy)\right)( \bar{X}_t^i)
\right]	\;ds .
\label{eq:eq935}
\end{align}
Since $h_0 \in \clc_b^2(\R^d)$, for the first term on the right side of \eqref{eq:eq935} we have
$$\Delta^i_t\cdot\left(\nabla Q_th_0(X^{i,N}_t) - \nabla Q_th_0(\bar X^{i}_t) \right)\le e^{-\alpha t} d \|\Hess h_0\|_{\infty} |\Delta^i_t|^2.
$$
Next note that for any $m \in \mathcal{P}(\R^d)$,
$$\nabla Q_{t-s}\Big( \int g(y-\cdot\,)\,m(dy)\Big)(x) = e^{-\alpha(t-s)} \int \nabla P_{t-s}g (y-x)\,m(dy).$$ 
For the second term on the right side in \eqref{eq:eq935} we will use the decomposition
\begin{align*}
& \int  \nabla P_{t-s}g (y-X^{i,N}_t)\,\mu^N_s(dy) - \int \nabla P_{t-s}g (y-\bar{X}_t^i)\,\mu_s(dy)  \\
&\quad=  \int  \nabla P_{t-s}g (y-X^{i,N}_t)\,\mu^N_s(dy) - \int \nabla P_{t-s}g (y-\bar{X}_t^i)\,\nu^N_s(dy)  \\
&\quad \quad + \int  \nabla P_{t-s}g (y-\bar{X}_t^i)\,\nu^N_s(dy) - \int \nabla P_{t-s}g (y-\bar{X}_t^i)\,\mu_s(dy),
\end{align*}
where $\nu^N$ is the empirical measure of $\{\bar X^i\}_{i=1}^N$, i.e. $\nu^N = \frac{1}{N} \sum_{i=1}^N \delta_{\bar X^i}$,
and $\nu^N_t$ is the marginal at time instant $t$.

From the above observations we have
\begin{align}
& d|\Delta^i_t|^2 \notag\\
& \quad \leq  \,\,2\,\Big[-v_*  \;+\;\chi\,e^{-\alpha\,t}d\,\|\Hess (h_0)\|_{\infty}\,\Big]\,|\Delta^i_t|^2\,dt  \notag\\
&\quad \quad +2\,\chi\,\beta\int_0^t\left(e^{-\alpha(t-s)} \Delta^i_t\cdot\left( \int  \nabla P_{t-s}g (y-X^{i,N}_t)\,\mu^N_s(dy) - \int \nabla P_{t-s}g (y-\bar{X}_t^i)\, \nu^N_s(dy) \right)\right)\,ds \;dt \notag\\
&\quad \quad+2\,\chi\,\beta\int_0^t e^{-\alpha(t-s)} \Delta^i_t\cdot \Big( \int  \nabla P_{t-s}g (y-\bar{X}_t^i)\,(\nu^N_s-\mu_s)(dy) \Big)\,ds\;dt, 	\label{E:UniformPOC_PP1}
\end{align}
where the above inequality is interpreted in the integral sense. That is, $d\phi_t\leq \psi_t\,dt$ is interpreted as $\phi_{b}-\phi_{a}\leq \int_a^b \psi_s\,ds$ for all $0\leq a\leq b$. In second term on the right of \eqref{E:UniformPOC_PP1}, the integrand has absolute value 
\begin{align}
&2 \chi \beta\frac{e^{-\alpha(t-s)}}{N}\Big| \sum_j \Delta^i_t\cdot\big[\,\nabla P_{t-s}g (X^{j,N}_s-X^{i,N}_t) -\nabla P_{t-s}g (\bar{X}_s^j-\bar{X}_t^i) \,\big]\,\Big|\notag\\
& \quad\leq\,2 \chi \beta\frac{e^{-\alpha(t-s)}\,d \|\Hess g\|_{\infty}}{N} \sum_j |\Delta^i_t|(|\Delta^i_t|+|\Delta^j_s|).\label{E:UniformPOC_PP2}
\end{align}
For the third term on the right of \eqref{E:UniformPOC_PP1}, we let 
$$a^{i,j}(s,t)\doteq \big(\,\nabla P_{t-s}g (\bar{X}_s^j-\bar{X}_t^i)  - \int \nabla P_{t-s}g (y-\bar{X}_t^i)\,\mu_s(dy) \,\big).$$
Then $\E[a^{i,j}(s,t)\,a^{i,k}(s,t)]=0$ for all $0\le s \le t$ whenever $j\neq k$ and hence
\begin{align}
\E\Big[\Big|  \int  \nabla P_{t-s}g (y-\bar{X}_t^i)\,(\nu^N_s-\mu_s)(dy)  \Big|^2\Big]&=\frac{1}{N^2}\E \Big[\Big|\sum_j a^{i,j}(s,t)\Big|^2 \Big] \nonumber \\
&= \frac{1}{N^2}\E \Big[\sum_{j,k} a^{i,j}(s,t)\cdot a^{i,k}(s,t) \Big]\nonumber\\
& = \frac{1}{N^2}\E \Big[\sum_{j} |a^{i,j}(s,t)|^2 \Big]\,
\leq \frac{2\,\|\nabla  g\|^2_{\infty}}{N}.\label{E:UniformPOC_PP3}
\end{align}

The proofs of Theorems \ref{T:POC_PP} and \ref{T:UniformPOC_PP}  will make use of the above calculations.\\

\begin{pf}{\it of Theorem \ref{T:POC_PP} }
	Fix $T \in (0,\infty)$.
	Letting $f^i(t)\doteq\sup_{s\in[0,t]}|\Delta^i_t|^2$, we have from \eqref{E:UniformPOC_PP1} on noting that $|v_*| \le L_{\nabla V}$,
	for all $T_1\in [0,T]$,
	\begin{align*}
	f^i(T_1) &\leq 
	2(\chi d \|\Hess h_0\|_{\infty} + L_{\nabla V})\int_0^{T_1} f^i(t) dt \\
	&\quad + \frac{2\chi \beta d\|\Hess g\|_{\infty}}{N} \sum_{j} \int_0^{T_1} \sqrt{f^i(t)} \int_0^t (\sqrt{f^i(t)} + \sqrt{f^j(s)}) ds\; dt\\
	&\quad + \frac{2\chi\beta}{N}\int_0^{T_1} \sqrt{f^i(t)} \int_0^t |\sum_j a^{i,j}(s,t)| ds\; dt\\
	&\le	\kappa\,\int_0^{T_1}\left[f^i(t) + t\,\Big(f^i(t)+\sqrt{f^i(t)}\frac{\sum_j\sqrt{f^j(t)}}{N}\Big) +\sqrt{f^i(t)} \int_0^t\frac{|\sum_j a^{i,j}(s,t)|}{N}\,ds\right]\;dt.
	\end{align*}
	Taking expectation, using Cauchy-Schwarz inequality, \eqref{E:UniformPOC_PP3} and the fact that
	$$\vartheta(t)\doteq\frac{\sum_{i}\E f^i(t)}{N}\geq \Big(\frac{1}{N}\sum_{i}\sqrt{\E f^i(t)}\Big)^2,$$
	we obtain
	\begin{align*}
	\vartheta(T_1) \leq \kappa_1\,\int_0^{T_1} \left[(1+2t)\,\vartheta(t) + \frac{t\,\sqrt{\vartheta(t)}}{\sqrt{N}}\right]\;dt
	\end{align*}
	and hence
	\begin{align*}
	\vartheta(T_1) \leq \kappa_{2,T}\,\int_0^{T_1}\left[\vartheta(t) + \frac{\sqrt{\vartheta(t)}}{\sqrt{N}}\right]\;dt
	\end{align*}
	for all $T_1\in[0,T]$. Note that $\Phi(z)\doteq z+\frac{\sqrt{z}}{\sqrt{N}}$, $z \in \R_+$, is an increasing function.	From Bihari's generalization of Gronwall's lemma (see Section 3 of \cite{B56})
	$$\vartheta(t) \le A^{-1} \left( A(0) + \kappa_{2,T} t\right), \mbox{ for all } t \in [0,T],$$
	where
	$$A(u) = \int_0^u \frac{ds}{\Phi(s)}, \; u\ge 0.$$
	Observing that
	$$A(u) = 2\log \left(\frac{N^{-1/2} + u^{1/2}}{N^{-1/2} }\right) \mbox{ and } A^{-1}(v) = \left((N^{-1/2} )(e^{v/2} - 1)\right)^2,$$
	we see that for all $t\in [0,T]$,
	$$\vartheta(t)\leq \bigg(\frac{(e^{\kappa_{2,T}\,t/2}-1)}{\sqrt{N}}\bigg)^2 \leq \frac{\kappa_{3,T}}{N}\quad\text{for }t\in[0,T].$$
	The proof is complete.
\end{pf}

\vspace{0.2in}

\begin{pf}{\it of Theorem \ref{T:UniformPOC_PP} }	
	For $i,j = 1, \cdots N$ and $0\le s \le t$
	\begin{equation}\label{E:UniformPOC_PP2a}
	\E[\,|\Delta^i_t|(|\Delta^i_t|+|\Delta^j_s|)\,]\leq \E|\Delta^i_t|^2 + \sqrt{\E|\Delta^i_t|^2\,\E|\Delta^j_s|^2}.
	\end{equation} 
	Combining  \eqref{E:UniformPOC_PP1}, \eqref{E:UniformPOC_PP2}, \eqref{E:UniformPOC_PP3} and \eqref{E:UniformPOC_PP2a}, we see that $f(t)\doteq \E |\Delta^i_t|^2$ satisfies  
	\begin{align}\label{E:UniformPOC_PP4}
	df(t)\, &\,\leq \,2\,\big( -v_*\,+\,\chi d\,\|\Hess h_0\|_{\infty}  \big)\,f(t)\; dt \notag\\
	&\qquad +2\,\chi\,\beta\,\int_0^te^{-\alpha(t-s)} d\|\Hess g\|_{\infty}\big(f(t)+\sqrt{f(t)\,f(s)} \,\big) \,ds \; dt \notag\\
	&\qquad+\frac{2\,\chi\,\beta}{N}\,\int_0^t e^{-\alpha(t-s)} \,\sqrt{f(t)\,2\,\|\nabla  g\|^2_{\infty}\,N}\;ds \; dt \notag\\
	&\,\leq \,2\,\big( -v_* \,+\chi C_1 \big)\,f(t)\; dt \notag\\
	&\qquad +2\,C_2\,\chi\,\beta\,\int_0^te^{-\alpha(t-s)}\big(f(t)+\sqrt{f(t)\,f(s)} \,\big) \,ds \; dt\notag\\
	&\qquad+\frac{4\,C_3\,\chi\,\beta}{\sqrt{N}}\,\sqrt{f(t)}\int_0^t e^{-\alpha(t-s)} \;ds \; dt, 
	\end{align}
	where $C_1 = d\|\Hess h_0\|_{\infty}$, $C_2 = d\|\Hess g\|_{\infty}$ and $C_3 = \|\nabla g\|_{\infty}$.
	Thus
	\begin{equation}\label{E:UniformPOC_PP5}
	df(t)\; \leq\; \left(-2\,\tilde \lambda\,f(t) + 2\,\tilde{C_2}\sqrt{f(t)}\int_0^te^{-\alpha(t-s)}\sqrt{f(s)}\,ds+\frac{2\,\tilde {C_3}}{\sqrt{N}}\,\sqrt{f(t)}\right) dt,
	\end{equation}
	where $-\tilde \lambda\doteq -v_*\,+\chi C_1 +C_2\chi\beta/\alpha$, $\,\tilde{C_2}\doteq C_2\,\chi\,\beta$ and $\tilde{C_3}\doteq 2\,C_3\,\chi\,\beta/\alpha$. 
	Under Assumption \ref{A:A} $\tilde \lambda \ge \tilde C_2/\alpha$. 
	Note that   $f(0)=0$ since $X^{i,N}_0=\bar{X}^i_0$. Multiplying both sides of \eqref{E:UniformPOC_PP5} by $e^{2\,\tilde\lambda\,t}$ and letting $\vartheta(t)\doteq e^{2\,\tilde\lambda\,t}\,f(t)$, we obtain
	\begin{equation}\label{E:ineqg}
	d\vartheta(t) \leq  2\sqrt{\vartheta(t)} \Big(\int_0^t\tilde{C_2}\,e^{(\tilde\lambda-\alpha)\,(t-s)}\sqrt{\vartheta(s)}\,ds+\frac{\tilde {C_3}e^{\tilde\lambda\,t}\,}{\sqrt{N}}\Big).
	\end{equation}
	In rest of the proof we estimate $\vartheta(t)$  using the above inequality.
	For this, heuristically, one can set $\zeta=\sqrt{\vartheta}$ to obtain a simplification
	\begin{equation*}
	\zeta'(t) \leq \tilde{C_2}\int_0^te^{(\tilde\lambda-\alpha)\,(t-s)}\zeta(s)\,ds+\frac{\tilde {C_3}e^{\tilde\lambda\,t}\,}{\sqrt{N}} \quad\text{on } \{ \zeta\neq 0\}.
	\end{equation*}
	Since we do not have any control for $\zeta'(t)$ on $\{\zeta=0\}$,
	we instead consider $\zeta_\epsilon(t)=\sqrt{\vartheta(t)+\epsilon^2}$ where $\epsilon>0$. Then $\vartheta'=2\zeta_\epsilon \zeta'_\epsilon$ and  $\sqrt{\vartheta}=\sqrt{\zeta_\epsilon^2-\epsilon^2}\leq \zeta_\epsilon$. Hence \eqref{E:ineqg} implies that
	$$d(\zeta_\epsilon (t))^2\leq 2\zeta_\epsilon(t)\Big(\int_0^t\tilde{C_2}\,e^{(\tilde\lambda-\alpha)\,(t-s)}\zeta_\epsilon(s)\,ds+ \frac{\tilde {C_3}e^{\tilde\lambda\,t}\,}{\sqrt{N}}\Big).$$
	Thus for a.e. $t\geq 0$,
	\begin{equation}\label{E:ineqh}
	\zeta'_\epsilon(t)\leq \int_0^t\tilde{C_2}\,e^{(\tilde\lambda-\alpha)\,(t-s)}\zeta_\epsilon(s)\,ds+ \frac{\tilde {C_3}e^{\tilde\lambda\,t}\,}{\sqrt{N}}.
	\end{equation}
	We will now use a comparison result for ordinary differential equations (ODE). Let $k_{\epsilon}$ be the solution of the ODE 
	\begin{equation}\label{E:ODEk}
	k_{\epsilon}''(t)= (\tilde \lambda-\alpha)\, k_{\epsilon}'(t) + \tilde{C_2}\,k_{\epsilon}(t) + \frac{\alpha\,\tilde C_3}{\sqrt{N}}\,e^{\tilde\lambda\,t},\; k_{\epsilon}(0)=\epsilon, \; k_{\epsilon}'(0)=\frac{\tilde {C_3}}{\sqrt{N}}.
	\end{equation}
	Note that the solution $k_{\epsilon}$ solves the integral equation
	\begin{equation}\label{E:eqt_k}
	k'_\epsilon(t)= \int_0^t\tilde{C_2}\,e^{(\tilde\lambda-\alpha)\,(t-s)}k_\epsilon(s)\,ds+ \frac{\tilde {C_3}e^{\tilde\lambda\,t}\,}{\sqrt{N}},\quad k_{\epsilon}(0)=\epsilon.
	\end{equation}
	%
	%
	It is straightforward to verify that the unique solution of \eqref{E:ODEk} converges to
	\begin{equation}\label{E:Sol_k}
	k(t) \doteq \frac{\tilde{C_3}}{(r_1-r_2)\sqrt{N}}\,\Big[e^{r_1t}-e^{r_2t}\,+\,\alpha\,\Big(
	\frac{e^{\tilde{\lambda}t}-e^{r_1t}}{\tilde{\lambda}-r_1}-\frac{e^{\tilde{\lambda}t}-e^{r_2t}}{\tilde{\lambda}-r_2}\Big)\Big]
	\end{equation}
	uniformly on compacts when $\epsilon \to 0$, where  $r_2<0<r_1$ are the zeros of the characteristic polynomial 
	\begin{equation}
	\theta(r) = r^2-(\tilde \lambda-\alpha)\,r - \tilde{C_2}.\label{eq:cheqn}
	\end{equation}
	In particular,
	\begin{equation}\label{E:bound_k}
	\sup_{t\leq T}\sup_{\epsilon\in(0,1)}|k_{\epsilon}(t)|<\infty\quad\text{for any }T\geq 0.
	\end{equation}
	Subtracting  \eqref{E:eqt_k} from \eqref{E:ineqh}, we see that $\phi_{\eps}=\zeta_{\epsilon}-k_{\epsilon}$ satisfies
	\begin{equation*}
	\phi_{\eps}(t) \leq \tilde{C_2} \int_0^t\int_0^{t_1}e^{(\tilde\lambda-\alpha)\,(t_1-t_2)}\phi_{\eps}(t_2)\,dt_2\,dt_1 ,\quad t\geq 0.
	\end{equation*}
	Upon iterating $M$ times, one has
	\begin{equation}\label{E:IterateM}
	\phi_{\eps}(t) \leq (\tilde{C_2})^M \int_{\{0\leq t_{2M}\leq t_{2M-1}\leq \cdots\leq t_1\leq t\}}\, \prod_{i=1}^M\,e^{(\tilde\lambda-\alpha)\,(t_{2i-1}-t_{2i})}\phi_{\eps}(t_{2M})\,dt_{2M}\cdots dt_2\,dt_1.
	\end{equation}
	For any $T>0$, $C_T=\sup_{t\leq T}\sup_{\epsilon\in(0,1)}|\phi_{\eps}(t)|<\infty$ by \eqref{E:bound_k} and Theorem \ref{T:POC_PP}. Hence the integrand of \eqref{E:IterateM} is at most
	$C_T\,e^{(\tilde\lambda-\alpha)\,T\,M}$.
	Thus \eqref{E:IterateM} implies that for every $T\in (0,\infty)$ there exists $\tilde C_T \in (0,\infty)$ such that for all $t\in [0,T]$,
	\begin{align*}
	\phi_{\eps}(t) 
	&\leq \tilde C_T^M\,\text{Volume}\,\{0\leq t_{2M}\leq t_{2M-1}\leq \cdots\leq t_1\leq t\}\\
	&\leq (\tilde C_T T^2)^M/(2M)!
	\end{align*}
	Letting $M\to\infty$, we see that $\sup_{t\in[0,T]}\phi_{\eps}(t) \leq 0$. Since $T>0$ is arbitrary, we have $\phi_{\eps}(t)\leq 0 $ for all $t\ge 0$. Thus $\zeta_{\epsilon}(t)\leq k_{\epsilon}(t)$ for all $t\ge 0$. Letting $\epsilon \to 0$, we obtain that $\sqrt{\vartheta}$ is bounded by $k$ defined by
	\eqref{E:Sol_k}. Recalling that $\vartheta(t)\doteq e^{2\,\tilde\lambda\,t}\,f(t)$, we have
	\begin{align}\label{E:sqrtf}
	\sqrt{f(t)}& \leq \frac{\tilde{C_3}}{(r_1-r_2)\sqrt{N}}\,\Big[e^{(r_1-\tilde\lambda)t}-e^{(r_2-\tilde\lambda)t}\,+\,\alpha\,\Big(
	\frac{1-e^{(r_1-\tilde\lambda)t}}{\tilde{\lambda}-r_1}-\frac{1-e^{(r_2-\tilde\lambda)t}}{\tilde{\lambda}-r_2}\Big)\Big]. 
	\end{align}
	Observe that Assumption \ref{A:A} implies that $\tilde C_2 < \tilde \lambda \alpha$ from which it follows that
	$$ \sqrt{(\tilde{\lambda}-\alpha)^2+4\tilde{C_2}} < \tilde \lambda + \alpha.$$ Recalling that $r_1, r_2$ are the positive and negative roots of \eqref{eq:cheqn}
	we now see that under Assumption \ref{A:A}, $\tilde{\lambda} >r_1>0>r_2$. Thus the right hand side of \eqref{E:sqrtf} is at most
	\begin{equation*}
	\frac{\tilde{C_3}}{(r_1-r_2)\sqrt{N}}\,\Big[1\,+\,
	\frac{\alpha}{\tilde{\lambda}-r_1}\Big] 
	\end{equation*}
	The proof is complete.	
\end{pf}

\vspace{0.2in}

We can now complete the proof of Corollary \ref{cor:UniformPOC_PE_2_1}.\\

\begin{pf}{\it of Corollary \ref{cor:UniformPOC_PE_2_1}. }	
	The first statement in the corollary is immediate from Theorem \ref{T:UniformPOC_PP} on noting that
	$$
	\mathcal{W}_2\Big(\mathcal{L}(X^{1,N}_t,\,X^{2,N}_t,\,\cdots,\,X^{k,N}_t),\;\mathcal{L}(\bar{X}^1_t)^{\otimes \,k}\Big)
	\le \left( \E \sum_{i=1}^k |X^{i,N}_t - \bar{X}_t^i|^2\right)^{1/2}.$$
	For the second statement note that since $\mu_0 \in \clp_q(\R^d)$, $\sup_{t\ge 0} \int_{\R^d} |x|^{\tilde q} \mu_t(dx) < \infty$ for some $\tilde q>2$(see  Remark \ref{Rk:Smoment}).
	Hence from Theorem 1.1 of \cite{FG15}, we have 
	\begin{equation*}
	\limsup_{N\to\infty}\,\sup_{t\geq 0} \E[\mathcal{W}^2_2(\nu_t^{N}, \mu_{t})]=0.
	\end{equation*}
	Also, from Theorem \ref{T:UniformPOC_PP}, as $N\to \infty$
	$$\sup_{t\geq 0} \E[\mathcal{W}^2_2(\mu_t^{N}, \nu^N_{t})] \le \sup_{t\geq 0}\E \frac{1}{N}\sum_{i=1}^N |X^{i,N}_t - \bar{X}_t^i|^2 \to 0.$$
	The result now follows on combining the above two displays and using the triangle inequality
	$$\mathcal{W}_2(\mu_t^{N}, \mu_{t}) \le \mathcal{W}_2(\mu_t^{N}, \nu^N_{t}) + \mathcal{W}_2(\nu_t^{N}, \mu_{t}).$$
\end{pf}


\subsection{Concentration bounds}
\label{sec:sec-concbd}
In this section we will first provide exponential concentration bounds that are uniform over compact time intervals. Under the stronger property in Assumption \ref{A:A} we will then show that  these bounds can be strengthened to
be uniform over the infinite time horizon. We begin with an upper bound for $\mathcal{W}_1(\mu_t^N,\mu_t)$ in terms of $\big(\mathcal{W}_1(\nu_s^N,\mu_s)\big)_{s\in[0,t]}$ where
$\nu^N$ is as introduced in \eqref{def:munuN}.

\subsubsection{Bounds in terms of empirical measures of independent variables}\label{s:4.3.1}

The following proposition  is a generalization of Proposition 5.1 in \cite{BGV05}.
Let $\tilde \lambda \doteq v_* -C_1\chi -C_2\chi\beta/\alpha$ be as above \eqref{E:ineqg}.
\begin{proposition}\label{prop:W1PP}
	For all $t\geq 0$  
	\begin{equation*}
	\mathcal{W}_1(\mu^N_t,\,\mu_t)\leq \mathcal{W}_1(\nu^N_t,\,\mu_t)+  \frac{\sqrt{C_2\,\chi\,\beta}}{2}\int_0^t\Big(
	e^{(\rr_1-\tilde{\lambda})(t-s)} - e^{(\rr_2-\tilde{\lambda})(t-s)}\Big)\mathcal{W}_1(\nu^N_s,\,\mu_s)\,ds,
	\end{equation*}
	where, as before,   $\rr_2<0<\rr_1$ are the solutions of $r^2-(\tilde \lambda-\alpha)\,r- C_2\,\chi\,\beta=0$.
\end{proposition}

\begin{proof}
	From \eqref{E:UniformPOC_PP1} and \eqref{E:UniformPOC_PP2} we have
	\begin{align}
	\,d|\Delta^i_t|^2
	& 	\leq\,\,2\,\Big[-v_* \;+\;e^{-\alpha\,t}\,\chi\,d\|\Hess h_0\|_{\infty}\,\Big]\,|\Delta^i_t|^2\,dt  \notag\\
	&\quad \, +2\,\chi\,\beta\int_0^t\left(\frac{e^{-\alpha(t-s)}\, d\|\Hess g\|_{\infty}}{N} \sum_j |\Delta^i_t|(|\Delta^i_t|+|\Delta^j_s|)\right)\,ds \;dt \notag\\
	&\quad \,+2\,\chi\,\beta\int_0^t \left(e^{-\alpha(t-s)} \Delta^i_t\cdot \Big(  \int  \nabla P_{t-s}g (y-\bar{X}_t^i)\,(\nu^N_s-\mu_s)(dy)  \Big)\right)\,ds\;dt. \label{E:Ineq_PP}
	\end{align}
	Instead of taking expectations as in Section \ref{sec:poc}, we now bound the third term on the right hand side of \eqref{E:Ineq_PP} using the inequality
	\begin{equation}\label{E:dualityPP}
	\Big|\int  \nabla P_{t-s}g (\bar{X}_t^i-y)\,(\nu^N_s-\mu_s)(dy)\Big| \leq d\|\Hess g\|_{\infty}\,\mathcal{W}_1(\nu^N_s,\,\mu_s)
	\end{equation}
	which follows from  the Kantorovich-Rubenstein duality \eqref{E:KR_Duality} and the fact that the Lipschitz norm of the function $x\mapsto \nabla_j P_{t-s}g(\bar{X}_t^i-x)$ is bounded by $d\|\Hess g\|_{\infty}$ for each $j = 1, \cdots , d$.
	
	Applying \eqref{E:dualityPP} to  \eqref{E:Ineq_PP}, then summing over $i$ and using the inequality $\sum_i|\Delta^i_t|\leq \sqrt{N\sum_i|\Delta^i_t|^2}$, we see that with 
	$F(t)\doteq \sum_{i=1}^N|\Delta^i_t|^2/N$
	%
	%
	\begin{align*}
	dF(t)\leq & \,2\,\Big[-v_*  \;+\;\,\chi\,d\|\Hess h_0\|_{\infty}\,\Big]\,F(t) \; dt \notag\\
	&\, +2\,\chi\,\beta\int_0^t\left(e^{-\alpha(t-s)}\, d \|\Hess g\|_{\infty}\,(F(t) + \sqrt{F(t)\,F(s)})\right)\,ds \; dt \notag\\
	&\,+2\,\chi\,\beta\int_0^t \left(e^{-\alpha(t-s)} \sqrt{F(t)}\,d \|\Hess g\|_{\infty}\,\mathcal{W}_1(\nu^N_s,\,\mu_s)\right)\,ds \; dt \notag\\
	\leq &\,2\,\Big[-v_*  \;+\,C_1\,\chi\,\Big]\,F(t)  \; dt \notag\\
	&\, +2\,\chi\,\beta\,C_2\int_0^te^{-\alpha(t-s)}\,(F(t) + \sqrt{F(t)\,F(s)})\,ds \;dt \notag\\
	&\,+2\,\chi\,\beta\,C_2\, \sqrt{F(t)}\int_0^t e^{-\alpha(t-s)}\,\mathcal{W}_1(\nu^N_s,\,\mu_s)\,ds \; dt \notag\\
	= &\,2\,\Big[-v_* \;+\,C_1\,\chi\,+\,\frac{\chi\,\beta\,C_2}{\alpha}\,\Big]\,F(t)  \; dt \notag\\
	&\, +2\,\chi\,\beta\,C_2\,\sqrt{F(t)}\int_0^t e^{-\alpha(t-s)} \left(\sqrt{F_s} +\mathcal{W}_1(\nu^N_s,\,\mu_s) \right)\,ds \;dt 
	\end{align*}
	Recalling that $\tilde \lambda  = v_* -C_1\chi -C_2\chi\beta/\alpha$,  we obtain  
	\begin{equation}\label{E:IneqF}
	dF(t)\; \leq\; \left(-2\,\tilde \lambda\,F(t) + 2\,\tilde{C_2}\sqrt{F(t)}\int_0^te^{-\alpha(t-s)}
	\left(\sqrt{F(s)}+ \mathcal{W}(s)\right)\,ds\right) dt,
	\end{equation}
	where $\mathcal{W}(s)\doteq\mathcal{W}_1(\nu^N_s,\,\mu_s)$ and  $\tilde{C_2}$ is as introduced below \eqref{E:UniformPOC_PP5}.
	Recall that $F_0=0$. 
	We now use \eqref{E:IneqF} to obtain an upper bound for $F(t)$ in terms of  $\{\mathcal{W}(s) \}_{s\in[0,t]}$. 
	
	As in \eqref{E:ineqg}, $G(t)\doteq e^{2\,\tilde \lambda\,t}\,F(t)$ satisfies
	\begin{equation*}
	G'(t) \leq 2\,\tilde{C_2} \sqrt{G(t)}
	\Big(\int_0^{t}e^{(\tilde \lambda-\alpha)\,(t-s)}\sqrt{G(s)}\,ds \,+ \,e^{(\tilde \lambda-\alpha)\,t}\int_0^{t}e^{\alpha s}\mathcal{W}(s)\,ds \Big).
	\end{equation*}	
	Following the same comparison argument as was used to obtain the bound for $\sqrt{\vartheta}$ ($\vartheta$ was introduced above \eqref{E:ineqg}), we let $H_\epsilon(t)=\sqrt{G(t)+\epsilon^2}$ where $\epsilon>0$ and obtain
	for a.e. $t\ge 0$
	\begin{equation}\label{E:Ineq_H_allt}
	H_\epsilon'(t) \leq \tilde{C_2} 
	\Big(\int_0^{t}e^{(\tilde \lambda-\alpha)\,(t-s)}H_\epsilon(s)\,ds \,+ \,e^{(\tilde \lambda-\alpha)\,t}\int_0^{t}e^{\alpha s}\mathcal{W}(s)\,ds \Big).
	\end{equation} 
	This time we need to solve  the inhomogeneous second order ODE
	\begin{equation*}
	K_{\epsilon}''(t)- (\tilde \lambda-\alpha)\, K_{\epsilon}'(t) -\tilde{C_2}\,K_{\epsilon}(t) = \tilde{C_2}\,e^{\tilde \lambda\,t}\mathcal{W}(t)
	\end{equation*}
	with initial conditions $K_{\epsilon}(0)=\epsilon$ and $K_{\epsilon}'(0)=0$. On solving this ODE, we obtain, as in  \eqref{E:Sol_k}
	and \eqref{E:sqrtf}, that $K_{\eps}$ converges uniformly on compacts as $\eps \to 0$ to $K$ defined as
	\begin{equation}\label{E:Sol_K}
	K(t) \doteq \frac{\tilde{C_2}}{(\rr_1-\rr_2)}
	\Big(
	e^{\rr_1\,t}\int_0^t
	e^{(\tilde \lambda-\rr_1)s}\mathcal{W}(s)\,ds - e^{\rr_2\,t}\int_0^te^{(\tilde \lambda-\rr_2)s}\mathcal{W}(s)\,ds\Big)
	\end{equation}
	Similar to the argument below \eqref{E:bound_k} we have that for all $\eps>0$ and $t\ge 0$
	$$\sqrt{G(t)+\eps^2} \le H_{\eps}(t) \le K_{\eps}(t).$$
	Sending $\eps \to 0$,
	\begin{align}
	\sqrt{F(t)} \le e^{-\tilde \lambda t} \sqrt{G(t)} \le e^{-\tilde \lambda t} K(t)= \frac{\tilde{C_2}}{(\rr_1-\rr_2)}\int_0^t\left(
	e^{(\rr_1-\tilde \lambda)(t-s)}\mathcal{W}(s) - e^{(\rr_2-\tilde \lambda)(t-s)}\mathcal{W}(s)\right)\,ds. \label{E:Sol_F}
	\end{align}
	
	On the other hand, 
	\begin{equation}\label{W1W2F}
	\mathcal{W}_1(\mu^N_t,\,\nu^N_t)\leq \mathcal{W}_2(\mu^N_t,\,\nu^N_t) \leq \Big(\sum_{i=1}^N|\Delta^i_t|^2/N\Big)^{1/2}=\sqrt{F(t)}.
	\end{equation}
	The desired equality now follows from the triangle inequality $\mathcal{W}_1(\mu^N_t,\,\mu_t)\leq \mathcal{W}_1(\nu^N_t,\,\mu_t)+ \mathcal{W}_1(\mu^N_t,\,\nu^N_t)$ and the observation
	that $\rr_1-\rr_2 \ge 2 \sqrt{\tilde C_2}$.
\end{proof}
The following corollary is an immediate consequence of the above proposition.
\begin{corollary}
	\label{cor:cor619}
	For every $T \in (0,\infty)$, there exists $C_T \in (0,\infty)$ such that
	\begin{equation*}
	\mathcal{W}_1(\mu^N_t,\,\mu_t)\leq \mathcal{W}_1(\nu^N_t,\,\mu_t)+  C_T \int_0^t \mathcal{W}_1(\nu^N_s,\,\mu_s)\,ds \quad\text{for }t\in[0,T].
	\end{equation*}
	Also, for all $t\ge 0$
	\begin{equation*}
	\mathcal{W}_1(\mu^N_t,\,\mu_t)\leq \mathcal{W}_1(\nu^N_t,\,\mu_t)+  \frac{\sqrt{C_2\,\chi\,\beta}}{2}\int_0^t
	e^{(r_1-\tilde{\lambda})(t-s)} \mathcal{W}_1(\nu^N_s,\,\mu_s)\,ds.
	\end{equation*}
\end{corollary}
Recall from Section \ref{sec:poc}  that under Assumption \ref{A:A} $(r_1-\tilde{\lambda})<0$. This will be key in obtaining a uniform in time bound from the last inequality in the corollary above.

\subsubsection{Moment bounds}\label{s:4.3.2}

Let $(\bar X_t)_{t\geq 0}$ be the nonlinear process solving \eqref{E:NonlinearProcessPP} and $\mu_t$ be its law at $t$. 
In Proposition  \ref{prop:Smoment} below we will give bounds on the square exponential moments of $\bar X_t$ under appropriate conditions.  The first part of the proposition holds under our standing assumptions along with a suitable integrability condition. For the second part we will make in addition an assumption that is weaker than Assumption \ref{A:A}, namely $v_*>0$, where $v_*$ was defined in \eqref{eq:eq951}.
%
Let, for $t, \theta \ge 0$,
$$S_{\theta}(t) \doteq \int_{\R^d} e^{\theta\,|x|^2}\,d\mu_t(x) = \E\big[e^{\theta\,|\bar X_t|^2}\big].$$
\begin{proposition}\label{prop:Smoment}
	Suppose $\mu_0$ is as in Theorem \ref{T:Concen_PP_longtime}, namely for some $\theta_0>0$, $S_{\theta_0}(0) <\infty$. Then
	\begin{enumerate}
		\item[(i)] For any $T\in (0,\infty)$, there exists $\theta_T>0$ such that
		$\sup_{t\in[0,T]}S_{\theta_T}(t) <\infty$.
		\item[(ii)] Suppose  that $v_*>0$. Then for any $\theta\in (0,\,\theta_0/4\,\wedge\,v_*/8)$, we have $\sup_{t\geq 0}S_{\theta}(t) < \infty$.
	\end{enumerate}
\end{proposition}

\bigskip 
\begin{pf} 
	(i) 
	Note that for all $t \in [0,T]$
	$$|\bar X_t|\leq |\bar X_0|+ \chi C T + \sup_{0\le t \le T }|B_t|+ L_{\nabla V}\int_0^t\,|\bar X_s|\,ds, \quad t\geq 0,$$
	where $C$ is as in Lemma \ref{L:Properties_PPN}.
	By Gronwall's lemma 
	$$\sup_{0\le t \le T }|\bar X_t| \leq a(T)e^{L_{\nabla V}T}$$ 
	where 
	$$a(T)=|\bar X_0|+ \chi C T + \sup_{0\le t \le T }|B_t|.$$
	Thus for $\theta>0$
	\begin{align*}
	\E\Big[ \sup_{0\le t \le T } e^{\theta |\bar X_t|^2} \Big]
	&\le \kappa_{1,T} \E\Big[ e^{2\theta |\bar X_0|^2}  e^{2\theta (\chi C T + \sup_{0\le t \le T }|B_t|)^2}\Big]\\
	&\le \kappa_{1,T} \left(\E\left[ e^{4\theta |\bar X_0|^2}\right]\right)^{1/2}\left(\E\left[ e^{4\theta (\chi C T + \sup_{0\le t \le T }|B_t|)^2}\right]\right)^{1/2}.
	\end{align*}
	Choose $\theta \le \frac{\theta_0}{4}$ such that
	$\E\left[ e^{8\theta  |B_T|^2)}\right] < \infty.$
	Then for all such $\theta$, $\sup_{0\le t \le T } S_{\theta}(t) < \infty$.
	
	%
	(ii) Fix $\theta >0$. We apply It\^o's formula to $\phi(x)=e^{\theta\,|x|^2}$. Note that $\nabla\phi(x)=2\theta e^{\theta\,|x|^2}x$ and $\Delta \phi(x)=2\theta e^{\theta\,|x|^2}\big(d+2\theta |x|^2\big)$. Thus for $t\ge 0$
	\begin{align}\label{E:Ito_S}
	d\phi(\bar X_t) =&\,  e^{\theta\,|\bar X_t|^2}\Big(2\theta \bar X_t\cdot (dB_t -\nabla V(\bar X_t)\,dt +\chi\nabla h(t,\bar X_t)\,dt) +\theta\big(d+2\theta |\bar X_t|^2\big)\,dt \Big).
	\end{align}
	We claim that with $\sigma \doteq \theta_0/4\,\wedge\,v_*/8$, for all $\theta \in (0, \sigma)$,
	$$M_{\theta}(t)\doteq\int_0^t e^{\theta\,|\bar X_s|^2} \bar X_s\cdot dB_s$$
	is a square integrable martingale. 
	Suppose for now that the claim is true. Then $\E[M_{\theta}(t)]=0$ for all $t\ge 0$ and $\theta \in (0, \sigma)$. 
	By our assumption   $-x\cdot \nabla V(x)\leq  -v_*|x|^2$ for all $x \in \R^d$. Also,  for any $\eta>0$, $|x| \leq \eta +\frac{|x|^2}{4\eta}$. Combining these observations with   Lemma \ref{L:Properties_PPN}, we obtain that
	\begin{equation}\label{E:S_kt}
	dS_{\theta}(t) \leq \E\Big[ e^{\theta\,|\bar X_t|^2}\big(A + B\,|\bar X_t|^2 \big) \Big] \;dt,
	\end{equation}
	for all $\eta>0$, where 
	\begin{align*}
	A = &\,\theta\big(2\chi\,\|\nabla h\|_{\infty}\eta + d\big) , \notag\\
	B = &\, 2\,\theta^2 -2b\,\theta\quad\text{and } b=v_*-\frac{\chi}{4\eta}\|\nabla h\|_{\infty}.\notag
	\end{align*}
	Since $v_*>0$, we can find $\eta \in (0, \infty)$ so that $b = v_*/2$ and consequently, since 
	$\sigma < v_*/2$, with this choice of $\eta$, $B<0$ for all $\theta \in (0, \sigma)$.  Therefore for such $\theta$
	\begin{equation}\label{E:S_kt2}
	dS_{\theta}(t) \leq \kappa \int_{\R^d}(R^2-|x|^2)\, e^{\theta\,|x|^2}\,d\mu_t(x),
	\end{equation}
	for some  $R,\,\kappa\in(0,\infty)$ depending only on $A$ and $B$. Decomposing the integrand on the right of \eqref{E:S_kt2} according to the size of $|x|$, one obtains
	$$dS_{\theta}(t) \leq \big(\kappa_1 -\kappa_2\,S_{\theta}(t)\big)\,dt$$
	where $\kappa_1,\,\kappa_2 \in (0,\infty)$. Since $\theta < \sigma \le \theta_0$, $S_{\theta}(0)<\infty$.
	A standard estimate now shows that $\sup_{t\geq 0}S_{\theta}(t) <\infty$ 
	for all $\theta\in (0,\sigma)$. 
	
	Finally we verify the claim. 
	Using the estimates $-x\cdot \nabla V(x)\leq  -v_*|x|^2$ and $|x| \leq \eta +\frac{|x|^2}{4\eta}$ once again, and choosing $\eta$ such that $b = v_*/2$ as before, we  have 
	by an application of It\^o's formula
	\begin{equation}\label{E:2ndmoment}
	|\bar X_t|^{2} \leq |\bar X_0|^2+ \int_0^t2 \bar X_s\cdot dB_s - \int_0^t v_*|\bar X_s|^2\,ds + \kappa_3\,t
	\end{equation}
	where $\kappa_3=2\chi \|\nabla h\|_{\infty} \eta + d$. Provided that $\sigma_1\leq v_*/4 \wedge \theta_0/2$, we can bound $S_{\sigma_1}(t)$ above by
	\begin{align}\label{E:Bound2}
	& e^{\sigma_1\,\kappa_3\,t\,}\,\E\Big[
	\exp{\Big(\sigma_1|\bar X_0|^2\Big)}\cdot
	\exp{\Big(\int_0^t2\sigma_1 \bar X_s\cdot dB_s -\int_0^t v_*\sigma_1 |\bar X_s|^2\,ds\Big)}
	\Big]\notag\\
	&\quad \leq  e^{\sigma_1\,\kappa_3\,t\,}\,\E\Big[
	\exp{\Big(\sigma_1|\bar X_0|^2\Big)}\cdot
	\exp{\Big(\int_0^t2\sigma_1 \bar X_s\cdot dB_s -\int_0^t 4\sigma_1^2 |\bar X_s|^2\,ds\Big)}
	\Big]\notag\\
	&\quad \le e^{\sigma_1\,\kappa_3\,t\,}\,\sqrt{\E\Big[
		\exp{\Big(2\sigma_1|\bar X_0|^2\Big)}\Big]\cdot\E\Big[
		\exp{\Big(\int_0^t4\sigma_1 \bar X_s\cdot dB_s -\int_0^t 8\sigma_1^2 |\bar X_s|^2\,ds\Big)}
		\Big]} \notag\\
	&\quad \le e^{\sigma_1\,\kappa_3\,t\,}\,\sqrt{\E\Big[
		\exp{\Big(2\sigma_1|\bar X_0|^2\Big)}\Big]}\notag\\
	&\quad \le \kappa_4 e^{\sigma_1\,\kappa_3\,t\,},
	\end{align}
	where the next to last inequality follows on noting that 
	\begin{equation*}
	\exp{\Big(\int_0^t4\sigma_1 \bar X_s\cdot dB_s -8\sigma_1^2\int_0^t | \bar X_s|^2\,ds\Big)}
	\end{equation*}
	is a supermartingale and  in the last inequality we have used the fact that since $\sigma_1 \le \theta_0/2$,
	$\E \exp (2\sigma_1 | \bar X_0|^2) < \infty$.
	Finally for $\theta < \sigma_1/2$ and $t< \infty$
	\begin{align*}
	\int_0^t \E\left[e^{2\theta |X_s|^2}|X_s|^2\right] ds &\le \frac{1}{\sigma_1 - 2\theta} \int_0^t \E e^{\sigma_1 |X_s|^2}\\
	&\le \frac{\kappa_4}{\sigma_1 - 2\theta}\int_0^t e^{\sigma_1\,\kappa_3\,s\,} ds\\
	&< \infty .
	\end{align*}
	This completes the proof of the claim and the result follows.
\end{pf}

\begin{remark}\label{Rk:Smoment} \rm
	(i)	In a similar manner it can be shown that if $\mu_0$ admits a finite square exponential moment of order $\theta_0>0$, then for every $T>0$ there is a $\theta_T>0$ such that
	$$\sup_{N,i}\sup_{0\le t\le T}\E\big[e^{\theta_T\,| X^{N,i}_t|^2}\big] < \infty .$$
	Furthermore if $v_*>0$, we have for all $\theta$ as in part (ii) of the above proposition
	$$\sup_{N,i}\sup_{t\ge 0}\E\big[e^{\theta\,| X^{N,i}_t|^2}\big] < \infty .$$
	(ii) Suppose that instead of assuming that $\mu_0$ has a finite squared exponential moment, we assume that
	$\mu_0 \in \clp_{q_0}(\R^d)$ for some $q_0 \ge 2$.  Then it follows easily that for any fixed $T<\infty$
	$$\sup_{0\le t \le T} \E|\bar X_t|^{q_0} < \infty.$$
	Furthermore, by applying It\^{o} formula to $|x|^q$ instead of $e^{\theta |x|^2}$ one can check that, if in addition $v_*>0$, with $q = \frac{q_0}{2}+1$,  
	$$\sup_{0\le t <\infty} \E|\bar X_t|^q < \infty.$$
	Also, analogous statements as in (i) hold with the squared exponential moment replaced by the $q$-th moment.
\end{remark}

\subsubsection{Time-regularity}\label{s:4.3.3}
In this section we give some estimates on the moments of the increments of the nonlinear process $\bar X$. These estimates are needed to appeal to results in \cite{BGV05} for proofs of our concentration bounds.

We start with moment estimates for $|\bar X_t-\bar X_s|$.
By Lemma \ref{L:Properties_PPN} and our assumption on $\nabla V$ we have
\begin{align*}
|\bar X_t-\bar X_s| &\,\leq |B_t-B_t|+ L_{\nabla V}\int_s^t|\bar X_r|\,dr+C\,(t-s),\quad 0\leq s \leq t<\infty,
\end{align*}
where  $C$  is as in Lemma \ref{L:Properties_PPN}.

Throughout this section we will assume that $S_{\theta_0}(0)<\infty$ for some $\theta_0 >0$.
Taking powers and using Proposition \ref{prop:Smoment}, we obtain the following result.
\begin{lemma}\label{L:moment_st}	
	For all $p\geq 1$ and $T>0$, there exists $C_{T,p}\in(0,\infty)$ such that
	\begin{equation}\label{E:moment_st}
	\E[|\bar X_t-\bar X_s|^{p}] \leq C_{T,p}\,|t-s|^{p/2}\quad\text{for }s,\,t\in[0,T].
	\end{equation}
	Moreover, if Assumption \ref{A:A} is satisfied, then  for some $C_p\in (0,\infty)$
	$$\E[|\bar X_t-\bar X_s|^{p}] \leq C_{p} (|t-s|^p+|t-s|^{p/2})\quad\text{for }s,\,t \ge 0.$$
\end{lemma}
Lemma \ref{L:moment_st} immediately implies 
\begin{equation}\label{E:W_st}
\mathcal{W}_p(\mu_t,\,\mu_s) \leq \tilde C_{T,p}\,|t-s|^{1/2} \quad\text{for }s,\,t\in[0,T]
\end{equation}
and under Assumption \ref{A:A}, for some $\tilde C_{p}>0$
$$\mathcal{W}_p(\mu_t,\,\mu_s) \leq \tilde C_{p} (|t-s| + |t-s|^{1/2}) \quad\text{for }s,\,t\ge 0.$$

Next we give an  estimate on the exponential moments of the increments.  
\begin{lemma}\label{L:exp_st}
	For all $T>0$, there exist  $\theta_T, C_T \in (0,\infty)$ such that
	\begin{equation*}
	\E\Big[\sup_{a\leq s,t\leq a+b}\exp(\theta_T|\bar X_t-\bar X_s|^2)\Big] \leq 1+C_T\,b
	\end{equation*}
	for all  $a,\,b\in[0,T]$.
\end{lemma}

\begin{pf}
	By an application of Cauchy-Schwarz inequality we see that it suffices to show that for some $\theta_T, C_T \in (0,\infty)$
	\begin{equation}\label{eq:eq842}
	\E\Big[\sup_{a\leq t\leq a+b}\exp(\theta_T|\bar X_t-\bar X_a|^2)\Big] \leq 1+C_T\,b
	\end{equation}
	for all $a,b\in [0,T]$.  Let $C$ be as in Lemma \ref{L:Properties_PPN} and $\clh_C$ be the class of all functions
	$v:[0,T]\times \R^d \to \R^d$ such that
	$$\sup_{t\ge 0} |v(t,x) - v(t,y)| \le C|x-y|,\; \sup_{t\ge 0} |v(t,x)| \le C \mbox{ for all } x,y \in \R^d.$$
	Note that from Lemma \ref{L:Properties_PPN}, for every $m \in \mathcal{P}\big(\mathcal{C}([0,\infty):\R^d)\big)$,
	$\nabla h^m \in \clh_C$.
	
	Given $v \in \clh_C$ and $z\in \R^d$, let $Y^{v,z}$ be the solution of the stochastic differential equation
	$$
	dY^{v,z}_t = dB_t - \nabla V(Y^{v,z}_t) + \chi v(t, Y^{v,z}_t)dt, \; Y^{v,z}_0 = z.$$
	By a standard conditioning argument it suffices to argue that for some $\theta_T, C_T \in (0,\infty)$
	\begin{equation}
	\label{eq:eq851}
	\sup_{z\in \R^d, v \in \clh_C} \E\left[ \sup_{0\le t \le b}\exp(\theta_T|Y^{v,z}_t-z|^2)\right] \leq 1+C_T\,b .
	\end{equation}
	Fix $(z,v) \in \R^d \times \clh_C$ and suppress it in the notation (i.e. write $Y^{v,z}$ as $Y$).
	Let $\theta:[0,T] \to \R$ be a non-negative  continuously differentiable function and write for $t \in [0, b]$
	$$Z_{t}\doteq e^{\theta(t)\,|Y_t-z|^2}.$$
	Using It\^o's formula we obtain, 
	\begin{align}\label{E:Ito_st}
	dZ_t &=\,  Z_t\Big[2\theta(t) (Y_t-z)\cdot (dB_t -\nabla V(Y_t)\,dt +\chi v(t,Y_t)\,dt) \notag\\
	&\qquad\qquad +\theta(t)\big(d+2\theta(t) |Y_t-z|^2\big)\,dt + \theta'(t)|Y_t-z|^2\,dt\Big].
	\end{align}
	Integrating, we obtain
	\begin{align}\label{E:Ito_st2}
	Z_{t} -1
	=&\, M_{t}+ \int_0^t Z_{r}\Big[2\theta(r) (Y_r-z)\cdot \big(-\nabla V(Y_r) +\chi v(r,Y_r)\big) \notag\\
	&\qquad\qquad +\theta(r)\big(d+2\theta(r) |Y_r-z|^2\big) + \theta'(r)|Y_r-z|^2\Big]\;dr,
	\end{align}
	where $M_{t}\doteq 2\int_0^t Z_{r} \theta(r) (Y_r-z)\cdot dB_r$.  
	From a similar argument as for the proof of Proposition \ref{prop:Smoment}(i) , there is a $\varsigma >0$
	such that
	\begin{equation}\label{eq:eq627}
	\sup_{(z,v) \in \R^d \times \clh_C}\sup_{0\le s \le T} \E e^{\varsigma |Y_s^{z,v}-z|^2} < \infty.
	\end{equation}
	One of the properties of  $\{\theta(t)\}_{0\le t \le T}$ chosen below will be that $\sup_{0\le t \le T} \theta(t) < \varsigma/2$.
	With such a choice of $\{\theta(t)\}$, $\{M_t\}$ is a martingale and consequently $\E(M_t)=0$ for all $t\ge 0$.
	
	%
	Applying Young's inequality we see  that for every $\eta >0$
	$$2 (Y_r-z)\cdot \big(-\nabla V(Y_r) +\chi v(r,Y_r)\big) \leq \eta |Y_r-z|^2 + \frac{|-\nabla V(Y_r) +\chi v(r,Y_r)|^2}{\eta},$$
	and thus
	\begin{align}\label{E:Ito_st3}
	Z_{t} 
	\leq &\, 1 +M_{t}+ \int_0^t Z_{r}\big(A_r + B_r\,|Y_r-z|^2\big)\;dr,
	\end{align}
	where 
	\begin{align*}
	A_r = &\,\theta(r)\,\Big(d+\frac{L_{\nabla V}|Y_r|+\chi C}{\eta}\Big) \notag\\
	B_r = &\,\eta\,\theta(r)+2\theta^2(r)+\theta'(r).\notag
	\end{align*}
	Rest of the argument is similar to \cite[Section 5.1]{BGV05} and so we only give a sketch.
	Choose $\theta(r)$ to be the solution of the ODE
	$$\eta\,\theta(r)+2\theta^2(r)+\theta'(r)=0$$
	with $\theta(0)$ to be a strictly positive and smaller than $\varsigma/2$.  It is easy to see that the solution is decreasing and strictly positive.
	Thus $B_r$ is identically zero and $ 0 < \theta(T)\le \theta(t) \le  \varsigma/2 $ for every $t\in [0,T]$.  
	As a consequence
	$$\E \sup_{0\le t \le b} Z_t \le 1 + \E \sup_{0\le t \le b} M_t + \int_0^b \E Z_r A_r dr.$$
	Using the bound in \eqref{eq:eq627} it is now checked exactly as in \cite{BGV05} that
	$$\sup_{(z,v) \in \R^d \times \clh_C}\sup_{0\le s \le T} \E Z_r A_r < \infty \mbox{ and }
	\sup_{(z,v) \in \R^d \times \clh_C}\E \sup_{0\le s \le b} M_s \le \tilde C_T b,$$
	for some $\tilde C_T < \infty$.  This proves \eqref{eq:eq851} and thus the result follows.
\end{pf}
\bigskip

With Lemma \ref{L:exp_st}, standard estimates for Brownian motion and the Chebyshev's exponential inequality then yields the following time-regularity of the empirical measures $\nu^N$. The proof is contained  in \cite[Section 5.2, page 577-578]{BGV05} and is omitted here.
\begin{proposition}\label{prop:nu_st}
	For any $T>0$, there exist $C_1,\,C_2\in(0,\infty)$ such that
	\begin{align*}
	\mathbb{P}\Big(\sup_{a\leq s,t\leq a+b}\mathcal{W}_1(\nu^N_s,\,\nu^N_t) >\epsilon\Big) \,\leq \exp\big(-N(C_1\,\epsilon^2-C_2\,b)\big) 
	\end{align*}
	for all $a,\,b\in[0,T]$ and $\epsilon>0$.
\end{proposition}

\begin{remark}
	Using Lemma \ref{L:Properties_PPN}, it is immediate that Lemma \ref{L:exp_st} and Proposition \ref{prop:nu_st} remain valid with the same constants if we replace $\bar X_t$ by $X^{N,i}_t$ and $\nu^N$ by $\mu^N$ respectively.
\end{remark}

\subsubsection{Proofs for the concentration bounds}

Equipped with the results Sections \ref{s:4.3.1}, \ref{s:4.3.2} and \ref{s:4.3.3} in the previous subsections, the proofs of Theorem \ref{T:Concen_PP} and Theorem \ref{T:Concen_PP_longtime} can be completed as in
\cite[Section 7.1]{BGV05} and \cite[Section 7.2]{BGV05} respectively. We only provide a sketch.

For Theorem \ref{T:Concen_PP}, note that first bound in Corollary \ref{cor:cor619} implies
\begin{equation*}
\sup_{t\in[0,T]}\mathcal{W}_1(\mu^N_t,\,\mu_t)\leq \kappa_T\,\sup_{t\in[0,T]}\mathcal{W}_1(\nu^N_t,\,\mu_t).
\end{equation*}
This in turn implies
\begin{equation*}
\mathbb{P}\Big(\sup_{t\in[0,T]}\mathcal{W}_1(\mu^N_t,\,\mu_t)>\epsilon \Big)  \leq \mathbb{P}\Big(\sup_{t\in[0,T]}\mathcal{W}_1(\nu^N_t,\,\mu_t)>\tilde{\epsilon} \Big),
\end{equation*}
where $\tilde{\epsilon}=\epsilon/\kappa_T$.
which is analogous to equation (75) in \cite{BGV05}.
Part (i) of Proposition \ref{prop:Smoment} guarantees that we can apply Theorem \ref{T:BGV1.1} to assert that for any $d'\in (d,\infty)$ and $\theta' \in(0,\theta)$, there exists a positive integer $N_0$ such that
\begin{equation*}
\sup_{t\in[0,T]}\mathbb{P}\Big(\mathcal{W}_1(\nu^N_t,\,\mu_t)>\tilde{\epsilon} \Big) \leq \exp{\big(-\kappa_{2,T}  \,\theta'\, N \, \eps^2\big)}
\end{equation*}
for all $\eps >0$ and $N \geq N_0\,\max(\eps^{-(d'+2)},\,1) $. 

To complete the proof of Theorem \ref{T:Concen_PP}, it remains to improve this estimate by interchanging $\sup_{t\in[0,T]}$ and $\mathbb{P}$. This ``exchange" can be achieved using the continuity estimates on $\nu^N_t$ and $\mu_t$ in Subsection \ref{s:4.3.3}. Details can be verified as in \cite[Section 7.1]{BGV05}. 
Specifically, Corollary 4.2, Proposition 4.3 and Proposition 5.1 in \cite{BGV05} are replaced by, respectively, part (i) of Proposition \ref{prop:Smoment},  \eqref{E:W_st} with $p=1$, and Proposition \ref{prop:nu_st}.

For Theorem \ref{T:Concen_PP_longtime}, we  start from the second bound in Corollary \ref{cor:cor619} and argue as in \cite[Section 7.2]{BGV05}. The key ingredient is
Proposition 4.1 in \cite{BGV05}. This result is   replaced by part (ii) of Proposition \ref{prop:Smoment} which gives uniform in time estimate for the square exponential moment for $\mu_t$. We omit the details.
%
%
%

\subsection{Uniform convergence of Euler scheme} 
\label{sec:secpfeuler}
In the proofs of Lemma \ref{L:Ybound} and Theorem  \ref{EScheme}, we need to solve
difference inequalities which are harder to handle than similar differential inequalities that appeared in the proofs of Theorem \ref{T:UniformPOC_PP} and Proposition \ref{prop:W1PP}.

\begin{proof}[Proof of Lemma \ref{L:Ybound}]
	From integration by parts in \eqref{G_theta} we see that
	\begin{equation*}
	G_\theta(\vec{x},y) =  \frac{-\chi\beta e^{-\alpha\theta}}{N}\sum_{j=1}^N\int_{\R^d} p(\theta,y, z)\nabla g(x_j-z)\,dz
	\end{equation*}
	and thus
	\begin{equation}\label{Gnorm}
	\sup_{N\geq 1}\sup_{\vec{x}\in \R^{dN}}\sup_{y\in\R^d}|G_\theta(\vec{x},y)| \le\chi\beta\, e^{-\alpha\theta}\|\nabla g\|_\infty.
	\end{equation}
	Also recalling that $V_t = V - \chi Q_t h_0$, 
	\begin{align}\label{Vtnorm}
	\sup_{t\ge0}|\nabla V_t(x)|&\le |\nabla V(x)| + \chi \sup_{t\ge0}|\nabla Q_th_0(x)|\le \kappa_1\,(|x|+1).
	\end{align}
	
	Recall from \eqref{EPPN2} that
	\begin{equation} 
	Y_{n+1}^{i} = Y_n^{i}+\Delta_nB^i + \epsilon\,\Big( \int_{0}^{n\epsilon}G_{n\epsilon-s}(\tilde{Y}_s^{(N),\epsilon},\,Y_n^{i})\,ds - \nabla V_{n\epsilon}(Y_n^{i})\Big)
	\end{equation}
	for $1\leq i\leq N$,
	where we have suppressed $N, \epsilon$ in the notation. Also note that $|\nabla (Q_th_0)(y)| \le  \|\nabla h_0\|_{\infty}$ for every $y \in \R^d$.
	Hence by \eqref{Gnorm} and \eqref{Vtnorm}, 
	\begin{align*}
	|Y_{n+1}^i|^2 \,\leq \,& -2\epsilon Y_n^i \cdot \nabla V_{n\epsilon} (Y_n^i)+ |Y_n^i|^2+|\Delta_nB^i|^2\\
	&\quad  + \kappa_2\,\epsilon^2(|Y_n^i|^2+1)+ \xi_n^i\cdot \Delta_nB^i + \kappa_3\,\epsilon |Y_n^i|
	\end{align*}
	where $\kappa_4 = \kappa_3 + 2\chi \|\nabla h_0\|_{\infty}$ and $\xi_n^i$ is measurable with respect to $\F_{n\epsilon}^{B^i} \doteq \sigma \{B^i_s: 0 \le s \le n\epsilon\}$. Note that $\E[\xi_n\Delta_nB^i] = 0$ for every $n,i$. Our assumption on $V$ then gives
	\begin{align*}
	\E |Y_{n+1}^i|^2 \le& -2\epsilon v_* \E|Y_n^i|^2 + \E|Y_n^i|^2 +\epsilon + \kappa_2\epsilon^2 \E|Y_n^i|^2 + \kappa_2\epsilon^2+ \kappa_4\epsilon \E|Y_n^i|.
	\end{align*}
	Let $a_n = \E|Y_n^i|^2$. Then for $\epsilon$ small enough, we have the nonlinear difference inequality
	\begin{equation}\label{Ybound}
	a_{n+1}\le (1-\epsilon v_*)a_n +\kappa\,\epsilon \sqrt{a_n} +2\epsilon
	\end{equation}
	Note that by our assumption $v_* >0$. By Young's inequlity
	%
	$\kappa\,\sqrt{a_n} \leq \eta\,a_n + \kappa^2/(4\eta)$ for all $\eta>0$. 
	Taking $\eta=v_*/2$ we obtain
	\begin{equation}\label{Ybound2}
	a_{n+1}-\delta\,a_n \le A
	\end{equation}	
	where $\delta=1-\frac{v_*\epsilon}{2} \text{ and } A=\epsilon\Big(\frac{\kappa^2}{2v_*}+2\Big)$.
	Note that $\delta\in(0,1)$ for $\epsilon>0$ small enough. Multiplying both sides of \eqref{Ybound2} by  $\delta^{-n}$, we obtain
	\begin{equation}\label{Ybound3}
	b_{n+1}-b_n \le A\,\delta^{-n} 
	\end{equation}	
	where $b_n=a_n\,\delta^{-(n-1)}$.
	Summing over $n$ then gives
	\begin{equation*}
	b_{n+1}-b_1 \le A\sum_{i=1}^n\delta^{-i}.
	\end{equation*}
	Since $a_0=0$ (giving $b_1\leq A$), we obtain
	\begin{equation*}
	a_{n+1} \le A\,\Big(\frac{1-\delta^{n+1}}{1-\delta}\Big) =\frac{2}{v_*}\Big(1-\big(1-v_*\epsilon/2\big)^{n+1}\Big)
	\Big(\frac{\kappa^2}{2v_*}+2\Big). 
	\end{equation*}	
	Thus we have
	\begin{equation}
	\sup_{n\ge0} a_n\leq \frac{2}{v_*}\Big(\frac{\kappa^2}{2v_*}+2\Big).
	\end{equation}
	The proof is complete.
\end{proof}

\begin{remark}\rm
	By \eqref{Gnorm} and \eqref{Vtnorm}, we also have (suppressing $N$ and $\eps$ in notation)
	\begin{align*}
	|Y^i_n-Y^i_{n-1}|&\le |\Delta_{n-1}B^i|+\epsilon\Big|\nabla V_{(n-1)\epsilon}(Y^i_{n-1})-\int_0^{(n-1)\epsilon}G_{(n-1)\epsilon-s}(\tilde{Y}_s,Y^i_{n-1})\,ds\Big|\\
	&\le |\Delta_{n-1}B^i| +\epsilon\Big(\kappa_1\,(|Y^i_{n-1}|+1)+\frac{\chi \beta\|\nabla g\|_\infty}{\alpha} \Big)\\
	&\le |\Delta_{n-1}B^i|+\epsilon\,\kappa_2\, (|Y^i_{n-1}|+1).
	\end{align*}
	From this and the uniform bounds in Lemma \ref{L:Ybound} we have that if $v_*>0$, 
	\begin{equation}\label{Yincp}
	\sup_{i,n}\E|Y^i_n-Y^i_{n-1}|^2\le \kappa_3\,\epsilon
	\end{equation}
	Similarly, with $X^{i,N}_t$ as in \eqref{PPN2}, 
	for $t\in[(n-1)\epsilon,n\epsilon]$,
	\begin{equation*}
	|X^{i,N}_t-X^{i,N}_{n\epsilon}|\le |B^i_t-B^i_{n\epsilon}|+\kappa_2\,\int_t^{n\epsilon}(|X^{i,N}_r|+1)\,dr 
	\end{equation*}
	and by the uniform bound for $\sup_{t\geq 0}\E|X^{i,N}_t|$ (see Remark \ref{Rk:Smoment}), 
	\begin{equation}\label{Xincp}
	\sup_{i,N}\E|X^{i,N}_t-X^{i,N}_{n\epsilon}|^2\le \kappa_3\,|t-n\epsilon|\quad\text{for all }n\geq 1,\, t \in [(n-1)\epsilon,n\epsilon].
	\end{equation}
\end{remark}

\begin{proof}[Proof of Theorem \ref{EScheme}]
	To simplify notations, we supress $N$ and $\epsilon$ to write $Y_n^i$ for $Y_n^{i,N,\epsilon}$ and $X_t^i$ for $X_t^{i,N}$. Denote $Z^i_n \triangleq Y_n^i-X_{n\epsilon}^i$ to be the error of the scheme.
	From \eqref{EPPN2} and \eqref{PPN2} we obtain \begin{equation} \label{Z_n}
	Z^i_n = Z^i_{n-1}+a^i_n+b^i_n\quad \text{for } n\geq 1,\,1\leq i\leq N,
	\end{equation}
	where 
	\begin{align*}
	a^i_n &= \int_{(n-1)\epsilon}^{n\epsilon} \Big(\int_0^{(n-1)\epsilon}G_{(n-1)\epsilon-s}(\tilde{Y}_s,Y^i_{n-1})\,ds - \int_0^t G_{t-s}(X_s,X^i_t)\,ds \Big)\,dt\\
	b^i_n &= \int_{(n-1)\epsilon}^{n\epsilon}\left[-\nabla V_{(n-1)\epsilon}(Y^i_{n-1})+\nabla V_t(X^i_t)\right]\,dt.
	\end{align*}
	From \eqref{Z_n} we have 
	$Z^i_n = \sum_{k=1}^n (a^i_k+b^i_k)$. 
	Hence \begin{equation} \label{Z_ninc}
	|Z^i_n|^2-|Z^i_{n-1}|^2 = (a^i_n+b^i_n)\,\big(Z^i_n+Z^i_{n-1}\big).
	\end{equation}
	
	{\bf Step (i): }
	We first estimate $|b^i_n\,Z^i_n+b^i_n\,Z^i_{n-1}|$.  	
	For this, we shall use the estimates
	\begin{align}
	|\nabla Q_th_0(y)- \nabla Q_th_0(x)|&\le e^{-\alpha t}d\|\Hess h_0\|_{\infty}|x-y|.\label{Qtxy}\\
	|\nabla Q_th_0(x)- \nabla Q_sh_0(x)|&\le \Big|e^{-\alpha t} \big(\nabla P_th_0(x)- \nabla P_sh_0(x)\big)+(e^{-\alpha t}-e^{-\alpha s})\nabla P_sh_0(x)\Big|\notag\\
	&\le d\,\|\Hess h_0\|_\infty e^{-\alpha t}\sqrt{t-s} \,+\, \|\nabla h_0\|_\infty \alpha\,e^{-\alpha\,(t\wedge s)}|t-s|. \label{Qtsy}
	\end{align}
	By \eqref{Qtxy}, \eqref{Qtsy} and our convexity assumption on $V$,
	\begin{align}\label{bnZn0}
	b^i_n\,Z^i_n &= \int_{ (n-1)\epsilon}^{n\epsilon} \Big [- \nabla V_{(n-1)\epsilon}(Y^i_{n-1})+\nabla V_{n\epsilon}(Y^i_n) \notag\\
	&\qquad\qquad\,\, +\nabla V_t(X^i_t)-\nabla V_{n\epsilon}(X^i_{n\epsilon}) -\nabla V_{n\epsilon}(Y^i_n) +\nabla V_{n\epsilon}(X^i_{n\epsilon})\Big]\,dt\cdot (Y^i_n-X^i_{n\epsilon}) \notag\\	
	&\le -v_*\,\epsilon\, |Z^i_n|^2 \,+ |Z^i_n|\int_{(n-1)\epsilon}^{n\epsilon} \Big(\chi d\,\|\Hess h_0\|_{\infty}\,|Z^i_n|\notag\\
	&\qquad\qquad\qquad\qquad +|-\nabla V(Y^i_{n-1})+\nabla V(Y^i_n)|+\kappa_1\,\Big[e^{-\alpha n\epsilon}|Y^i_n-Y^i_{n-1}|+\sqrt{\epsilon}\Big]\notag\\
	&\qquad\qquad\qquad\qquad +|\nabla V(X^i_{t})-\nabla V(X^i_{n\epsilon})|+\kappa_1\,\Big[e^{-\alpha n\epsilon}|X^i_t-X^i_{n\epsilon}|+\sqrt{\epsilon}\Big]\Big)\,dt \notag\\		
	&\le -v_*\,\epsilon\, |Z^i_n|^2 \,+ |Z^i_n|\int_{(n-1)\epsilon}^{n\epsilon} \Big(\chi d\,\|\Hess h_0\|_{\infty}\,|Z^i_n|\notag\\
	&\qquad\qquad\qquad\qquad + L_{\nabla V} |Y^i_n-Y^i_{n-1}|+\kappa_1\,\Big[e^{-\alpha n\epsilon}|Y^i_n-Y^i_{n-1}|+\sqrt{\epsilon}\Big]\notag\\
	&\qquad\qquad\qquad\qquad +L_{\nabla V}|X^i_t-X^i_{n\epsilon}|+\kappa_1\,\Big[e^{-\alpha n\epsilon }|X^i_t-X^i_{n\epsilon}|+\sqrt{\epsilon}\Big]\Big)\,dt \notag\\
	&\le\big(-v_*+ \chi d\,\|\Hess h_0\|_{\infty}\big)\,|Z^i_n|^2\epsilon\,+ 2\kappa_1\,\epsilon^{3/2}|Z^i_n|\notag\\
	&\qquad +\eps \Big(L_{\nabla V}+\kappa_1\,e^{-\alpha n\epsilon}\Big)|Z^i_n|\,|Y^i_n-Y^i_{n-1}|\notag\\
	&\qquad +\Big(L_{\nabla V}+\kappa_1\,e^{-\alpha n\epsilon}\Big)|Z^i_n| \int_{ (n-1)\epsilon}^{n\epsilon} |X^i_t-X^i_{n\epsilon}|\,dt.
	\end{align}
	Taking expectations in \eqref{bnZn0} and using \eqref{Yincp} and \eqref{Xincp} we obtain
	\begin{equation}\label{bnZn}
	\E[b^i_n\,Z^i_n]\le \big(-v_*+ \chi d\,\|\Hess h_0\|_{\infty}\big) \,\epsilon\,\E |Z^i_n|^2 +\kappa_2\, \epsilon^{3/2}\,\sqrt{\E|Z^i_n|^2}.
	\end{equation}
	
	The same argument gives
	\begin{equation}\label{bnZn-1}
	\E[b^i_n\,Z^i_{n-1}]\le \big(-v_*+ \chi d\,\|\Hess h_0\|_{\infty}\big) \,\epsilon\,\E |Z^i_{n-1}|^2 +\kappa_2\, \epsilon^{3/2}\,\sqrt{\E|Z^i_{n-1}|^2}.
	\end{equation}
	
	{\bf Step (ii) }
	Next we estimate $|a^i_n|$. By \eqref{Gnorm}, 
	\begin{equation*}
	|a^i_1| = \Big|\int_0^\epsilon\int_0^t G_{t-s}(X_s,X^i_t)\,ds\,dt\Big|\le \chi\beta\, \|\nabla g\|_\infty \frac{\epsilon^2}{2}.
	\end{equation*}
	For $n\ge2$, by \eqref{Gnorm} again,
	\begin{align}
	|a^i_n|&=\Big|\int_{ (n-1)\epsilon}^{n\epsilon}\Big( \int_0^{(n-1)\epsilon} \Big (G_{(n-1)\epsilon-s}(\tilde{Y}_s,Y^i_{n-1})-G_{t-s}(X_s,X^i_t)\Big)\,ds \notag \\
	&\qquad\qquad\quad- \int_{ (n-1)\epsilon}^t G_{t-s}(X_s,X^i_t)\,ds\Big)\,dt\,\Big|\notag\\
	&\le \Big|\int_{ (n-1)\epsilon}^{n\epsilon} A_n^{(1,i)}(t)+A_n^{(2,i)}(t)\,dt\Big| + \chi\beta\,\|\nabla g\|_\infty \frac{\epsilon^2}{2},\label{an}
	\end{align} 
	where
	\begin{align*}
	A_n^{(1,i)}(t) &= \int_0^{(n-1)\epsilon}\Big(G_{(n-1)\epsilon-s}(X_s,X^i_t) - G_{t-s}(X_s,X^i_t)\Big)\,ds\\
	A_n^{(2,i)}(t) &= \int_0^{(n-1)\epsilon} \Big(G_{(n-1)\epsilon-s}(\tilde{Y}_s,Y^i_{n-1})-G_{(n-1)\epsilon-s}(X_s,X^i_t)\Big)\,ds.
	\end{align*}
	
	For $A_n^{(1,i)}(t)$, using that $g \in \clc_b^2(\R^d)$, we have for $\theta_1\le\theta_2$, $|\theta_1-\theta_2| \le \epsilon$
	\begin{align*}
	&\quad|G_{\theta_1}(\vec{x},y)-G_{\theta_2}(\vec{x},y)|\\
	&\le \frac{ \chi\beta\,e^{-\alpha\theta_1}}{N}\sum_{j=1}^N\int_{ \R^d} \big(p(\theta_1,y,z)-p(\theta_2,y,z)\big) \nabla g(x_j-z)\,dz +  \chi\beta\,\|\nabla g\|_\infty |e^{-\alpha \theta_1}-e^{-\alpha \theta_2}|\\
	&\le \kappa_3\,\sqrt{\epsilon}\, e^{-\alpha \theta_1}.
	\end{align*}
	Putting $\theta_1 = (n-1)\epsilon-s$ and $\theta_2 = t-s$, we obtain
	\begin{equation}\label{A1n}
	|A_n^{(1,i)}(t)|\le \kappa_3\,\sqrt{\epsilon} \int_0^{(n-1)\epsilon} e^{-\alpha((n-1)\epsilon-s)}\,ds\le \frac{\kappa_3}{\alpha}\sqrt{\epsilon} .
	\end{equation}
	
	For $A_n^{(2,i)}(t)$, note that for $s \in [0, (n-1)\eps]$, $t \in [(n-1)\eps, n\eps]$
	\begin{align*}
	&\,|G_\theta(\tilde{Y}^j_s,Y^i_{n-1})- G_\theta(X_s,X^i_t)|\notag\\
	=& \,\frac{ \chi\beta\,e^{-\alpha\theta}}{N}\sum_{j=1}^N\int_{ \R^d} \Big (p(\theta,Y^i_{n-1}-z)\nabla g(\tilde{Y}^j_s-z)-p(\theta,X^i_t-z)\nabla g(X^j_s-z)\Big)\,dz\\
	=&\, \frac{ \chi\beta\,e^{-\alpha\theta}}{N}\sum_{j=1}^N\int p(\theta,X^i_t-z)\Big[ \nabla g(\tilde{Y}^j_s-Y^i_{n-1}+X^i_t-z)-\nabla g(X^j_s-z)\Big]\,dz\\
	\le &\, \frac{ \chi\beta\,e^{-\alpha\theta}}{N}\sum_{j=1}^N d\|\Hess g\|_\infty (|\tilde{Y}^j_s-X^j_s+X^i_t-Y^i_{n-1}|)
	\end{align*}
	where we have used the substitution $z\mapsto Y^i_{n-1}-X^i_t+w$ in the first integral in the first equality.
	When $s\in[(k-1)\epsilon,k\epsilon)$, we have $\tilde{Y}_s = Y_{k-1}$ and 
	\begin{align*}
	|\tilde{Y}^j_s- X^j_s+& X^i_t-Y^i_{n-1}| \le |Z^j_{k-1}|+|Z^i_n|+ \Err_{s,t}(i,j), 
	\end{align*}	
	where $\Err_{s,t}(i,j)\doteq |X^j_{(k-1)\epsilon}-X^j_s|+|X^i_t-X^i_{n\epsilon}|+|Y^i_{n}-Y^i_{n-1}|$ is an error term which, in view of  \eqref{Yincp} and \eqref{Xincp} satisfies,
	\begin{equation}\label{Errij}
	\E[\Err_{s,t}(i,j)^2] \leq \kappa_4\,\epsilon.
	\end{equation}
	
	From the above calculations and recalling that $C_2 = d \|\Hess g\|_{\infty}$
	\begin{align}\label{A2n}
	&\,|A_n^{(2,i)}(t)| \,= \Big|\sum_{k=1}^{n-1}\int_{ (k-1)\epsilon}^{k\epsilon} G_{(n-1)\epsilon-s}(Y_{k-1},Y^i_{n-1})-G_{(n-1)\epsilon-s}(X_s,X^i_t)\,ds\Big|\notag\\
	&\le \frac{ \chi\beta\,C_2}{N}\sum_{j=1}^N \sum_{k=1}^{n-1}\int_{(k-1)\epsilon}^{k\epsilon}e^{-\alpha((n-1)\epsilon-s)}\,\big(|Z^j_{k-1}|+|Z^i_n|+ \Err_{s,t}(i,j)\big)\,ds.
	\end{align}
	Using \eqref{A1n} and \eqref{A2n} in \eqref{an}, 
	\begin{align}\label{an2}
	|a^i_n|\le & \,\kappa_5\, \epsilon^{3/2} \,+\,\frac{ \chi\beta\,C_2}{N}\int_{(n-1)\epsilon}^{ n\epsilon} \Big(\sum_{k=1}^{n-1}\int_{(k-1)\epsilon}^{k\epsilon} e^{-\alpha ((n-1)\epsilon-s)}\notag\\
	& \qquad\qquad\qquad\qquad\qquad\sum_{j=1}^N\big(|Z^j_{k-1}|+|Z^i_n|+ \Err_{s,t}(i,j)\big)\,ds\Big) dt.
	\end{align}
	
	{\bf Step (iii) } We now combine steps (i) and (ii) to obtain an inequality for 
	$$f_n\doteq \frac{1}{N}\sum_{i=1}^N\E|Z^i_n|^2.$$
	This inequality is the discrete analogue of a differential inequality similar to \eqref{E:UniformPOC_PP5}.
	By exchangeability we have
	\begin{equation}
	\label{eq:eq1018}
	f_n = \E|Z^i_n|^2, \; i = 1, \cdots, N.
	\end{equation}
	By Cauchy-Schwartz inequality, the fact 
	$\frac{1}{N}\sum_{i=1}^N\sqrt{\E|Z^i_n|^2}\leq \sqrt{f_n}$ and the bound \eqref{Errij},
	\begin{align}
	&\frac{1}{N^2}\sum_{i=1}^N\sum_{j=1}^N\E \Big[|Z^i_n|\,\big(|Z^j_{k-1}|+|Z^i_n|+ \Err_{s,t}(i,j)\big)\Big]\notag\\
	\leq &\,f_n + \Big(\frac{1}{N}\sum_{i=1}^N\sqrt{\E|Z^i_n|^2}\Big)\Big(\frac{1}{N}
	\sum_{i=1}^N\sqrt{\E|Z^i_{k-1}|^2}\Big)+ \frac{1}{N}\sum_{i=1}^N\sqrt{\E|Z^i_n|^2}\,\left(\frac{1}{N}\sum_{j=1}^N\E|\Err_{s,t}(i,j)|^2\right)^{1/2} \notag\\
	\leq &\,f_n+\sqrt{f_n\,f_{k-1}} +\kappa_6\,\sqrt{\epsilon}\,\sqrt{f_n}.\label{anzn_old}
	\end{align}
	Let
	\begin{equation}\label{signma_nk}
	\sigma_{n,k}\doteq\int_{(k-1)\epsilon}^{k\epsilon}e^{-\alpha((n-1)\epsilon-s)}\,ds=\frac{e^{-\alpha\epsilon\,n}}{\alpha}\big(e^{\alpha\epsilon(k+1)}-e^{\alpha\epsilon\,k}\big)
	\end{equation}
	Then $\sum_{k=1}^{n-1}\sigma_{n,k}=(1-e^{-\alpha(n-1)\epsilon})/\alpha\leq 1/\alpha$ and
	from \eqref{an2} and \eqref{anzn_old} 
	we have
	\begin{align}\label{sum_anzn}
	\frac{1}{N}\sum_{i=1}^N\E[|a^i_n\,Z^i_n|]
	\leq \frac{ \chi\beta\,C_2}{\alpha}\,\epsilon\,f_n+  \chi\beta\,C_2\,\epsilon\,\sqrt{f_n}\sum_{k=1}^{n-1}\sigma_{n,k}\,\sqrt{f_{k-1}}\,+
	\kappa_7\,\epsilon^{3/2}\sqrt{f_n}.	
	\end{align}
	Similarly,
	\begin{align}\label{sum_anzn-1}
	\frac{1}{N}\sum_{i=1}^N\E[|a^i_n\,Z^i_{n-1}|]
	\leq \frac{ \chi\beta\,C_2}{\alpha}\,\epsilon\,f_{n-1}+  \chi\beta\,C_2\,\epsilon\,\sqrt{f_{n-1}}\sum_{k=1}^{n-1}\sigma_{n,k}\,\sqrt{f_{k-1}}\,+\kappa_7\,\epsilon^{3/2}
	\sqrt{f_{n-1}}.	
	\end{align}
	
	Therefore, summing over $i$ in \eqref{Z_ninc} and then using \eqref{bnZn}, \eqref{bnZn-1} \eqref{sum_anzn} and \eqref{sum_anzn-1}, we obtain a nonlinear ``difference-summation" inequality
	
	\begin{align}\label{DiscreteDE}
	f_n-f_{n-1}\le\; -\tilde \lambda \,\epsilon \big(f_n+f_{n-1}\big) &+ \tilde{C_2}\,\epsilon\big(\sqrt{f_n}+\sqrt{f_{n-1}}\big)\,\sum_{k=1}^{n-1}\sigma_{n,k}\,\sqrt{f_{k-1}} \notag\\
	&+\kappa_8\,\epsilon^{3/2}\big(\sqrt{f_n}+\sqrt{f_{n-1}}\big),	
	\end{align}
	where $\sigma_{n,k}$ is defined in \eqref{signma_nk}, and as before, $-\tilde \lambda\doteq -v_*\,+\chi C_1 +C_2\chi\beta/\alpha$ and $\,\tilde{C_2}\doteq C_2\,\chi\,\beta$.
	We can rewrite \eqref{DiscreteDE} as
	\begin{align} \label{DiscreteDE2ab}
	\frac{f_n-f_{n-1}}{\epsilon}\le& -\tilde \lambda \,\big(f_n+f_{n-1}\big) + \,\big(\sqrt{f_n}+\sqrt{f_{n-1}}\big)\,
	\Big(
	\tilde{C_2}\,\sum_{k=1}^{n-1}\sigma_{n,k}\,\sqrt{f_{k-1}} +\kappa_8\,\sqrt{\epsilon}\,\Big)	
	\end{align}	
	which is the discrete analogue of a differential inequality similar to \eqref{E:UniformPOC_PP5}. 
	
	{\bf Step (iv)}
	Finally, we use \eqref{DiscreteDE} to obtain a uniform (in $N,\,n,\,\epsilon$) upper bound for $f_n$, under Assumption \ref{A:A}. Note that under this assumption $\tilde \lambda > 0$. 
	Let
	$$\delta\doteq\frac{1-\tilde \lambda\,\epsilon}{1+\tilde \lambda\,\epsilon}.$$
	Note that $\delta \in (0,1)$ for $\epsilon$ small enough.
	Moving the negative term $ -\tilde \lambda \,\epsilon \big(f_n+f_{n-1}\big)$ to the other side of the inequality in
	\eqref{DiscreteDE2ab} and then multiplying both sides by  $\delta^{-(n-1)}$ and letting $g_n=f_n/\delta^{n-1}$
	we obtain 
	\begin{align}\label{g_n}
	g_n-g_{n-1} \leq& \,\frac{\epsilon}{(1+\tilde \lambda\,\epsilon)\,\delta^{n-1}}\,\big(\sqrt{f_n}+\sqrt{f_{n-1}}\big)\,
	\Big(
	\tilde{C_2}\,\sum_{k=1}^{n-1}\sigma_{n,k}\,\sqrt{f_{k-1}} +\kappa_8\,\sqrt{\epsilon}\,\Big)\notag\\
	\leq &\, \frac{\epsilon}{(1+\tilde \lambda\,\epsilon)\,\delta^{n/2}}\,\big(\sqrt{g_n}+\sqrt{g_{n-1}}\big)\,
	\Big(
	\tilde{C_2}\,\sum_{k=1}^{n-1}\sigma_{n,k}\,\delta^{(k-2)/2}\,\sqrt{g_{k-1}} +\kappa_8\,\sqrt{\epsilon}\,\Big)
	\end{align}
	where  the second inequality follows on noting that $\sqrt{f_n}+\sqrt{f_{n-1}}\leq \delta^{(n-2)/2}(\sqrt{g_n}+\sqrt{g_{n-1}})$ for $n\geq 2$. Similar to the proof of Theorem \ref{T:UniformPOC_PP} we consider a small positive perturbation of $g_n$ and let  $h_{\theta}(n)\doteq\sqrt{g_n+\theta^2}$ where $\theta>0$. The inequality in \eqref{g_n} then implies
	\begin{align*}
	h^2_{\theta}(n)-h^2_{\theta}(n-1)
	\leq &\, \frac{\epsilon}{(1+\tilde \lambda\,\epsilon)\,\delta^{n/2}}\,\big(h_{\theta}(n)+h_{\theta}(n-1)\big)\,
	\Big(
	\tilde{C_2}\,\sum_{k=1}^{n-1}\sigma_{n,k}\,\delta^{(k-2)/2}\, h_{\theta}(k-1) +\kappa_8\,\sqrt{\epsilon}\,\Big).
	\end{align*}	
	Since $h_{\theta}$ is strictly positive, we  obtain
	\begin{align}\label{h_n}
	h_{\theta}(n)-h_{\theta}(n-1)
	\leq &\, \frac{\epsilon}{(1+\tilde \lambda\,\epsilon)\,\delta^{n/2}}\,
	\Big(
	\tilde{C_2}\,\sum_{k=1}^{n-1}\sigma_{n,k}\,\delta^{(k-2)/2}\,h_{\theta}(k-1) +\kappa_8\,\sqrt{\epsilon}\,\Big).
	\end{align}	
	Consider the recursion equation obtained by replacing the inequality in \eqref{h_n} by equality, namely
	\begin{align}\label{h_n2}
	k_n-k_{n-1}
	= &\, \,
	A\,\Big(\sum_{i=1}^{n-1}\sigma_{n,i}\,\delta^{(i-2-n)/2}\,k_{i-1}\Big) +\frac{B}{\delta^{n/2}},
	\end{align}
	and $k_0=h_{\theta}(0)=\theta$, where 
	\begin{equation}\label{eq:eq1046}
	A=\frac{\epsilon\,\tilde{C_2}}{1+\tilde \lambda\,\epsilon}=O(\epsilon)\quad \text{and}\quad B= \frac{\epsilon^{3/2}\kappa_8}{1+\tilde \lambda\,\epsilon}=O(\epsilon^{3/2}).
	\end{equation}
	
	By evaluating $\Delta_n-\frac{e^{-\alpha \epsilon}}{\delta^{1/2}}\Delta_{n-1}$ where $\Delta_{n-1}\doteq k_n-k_{n-1}$, we can convert the above equation \eqref{h_n2} to the following second order linear difference equation:
	\begin{equation}\label{SecondDifference}
	\begin{cases}
	k_{n+1}- p\, k_n +q\,k_{n-1} =A_n,\;n\geq 2,\\
	\quad k_0=\theta,\quad k_1=\theta+\frac{B}{\delta^{1/2}},
	\end{cases}
	\end{equation}
	where 
	$$p=1+\frac{e^{-\alpha\epsilon}}{\delta^{1/2}}\to 2,\; q=\frac{e^{-\alpha\epsilon}}{\delta^{1/2}}- \frac{A\,(1-e^{-\alpha\epsilon})}{\alpha\,\delta^{3/2}}\to 1\,\text{ as }\epsilon\to 0\text{ and }\, A_n=
	\frac{B(1-e^{-\alpha\epsilon})}{\delta^{(n+1)/2}}.$$
	Solution of \eqref{SecondDifference} can be explicitly given as
	\begin{equation}\label{SecondDifference_sol}
	k_n=c^{(\theta)}_1\,r_1^n +c^{(\theta)}_2r_2^n +c_3\,\delta^{-n/2}
	\end{equation}	
	where $r_1>r_2>0$ are the distinct positive real roots of $r^2-pr+q=0$, and
	\begin{equation}\label{c_theta}
	c^{(\theta)}_2=\frac{r_1(\theta-c_3)-\theta-\delta^{-1/2}(B-c_3)}{r_1-r_2}, \; c^{(\theta)}_1=\theta-c_3-c^{(\theta)}_2  \text{ and }c_3 = \frac{B(1-e^{-\alpha\epsilon})}{1-p\,\delta^{1/2}+q\,\delta}.
	\end{equation}
	On other hand, induction easily gives $h_{\theta}(n)\leq k_n$ for all $n\geq 0$ and $\theta>0$.
	
	Thus we have for all $\theta >0$
	\begin{align*}
	\sqrt{f_n}=\delta^{(n-1)/2}\,\sqrt{g_n}=\delta^{(n-1)/2}\,\sqrt{h^2_{\theta}(n)-\theta^2}\leq \delta^{(n-1)/2}\,k_n.
	\end{align*}
	Sending $\theta\to 0$ in \eqref{c_theta} we obtain
	\begin{align*}
	\sqrt{f_n}&\,\leq \frac{c_1}{\delta^{1/2}}\,(\delta^{1/2}\,r_1)^n +\frac{c_2}{\delta^{1/2}}\,(\delta^{1/2}\,r_2)^n+ \frac{c_3}{\delta^{1/2}}
	\end{align*}	
	where 
	\begin{equation}\label{c_0}
	c_2=\frac{(\delta^{-1/2}-r_1)c_3 -\delta^{-1/2}B}{r_1-r_2}\quad \text{and}\quad c_1=-c_3-c_2. 
	\end{equation}
	We claim that for some $\kappa_9, \kappa_{10}\in (0,\infty)$
	\begin{equation}\label{claim}
	0<c_3 < \kappa_9\sqrt{\epsilon},\;
	r_1-r_2 > \kappa_9 \epsilon\quad \text{and}\quad 0<1-\delta^{1/2}\,r_1 < \kappa_9\epsilon
	\end{equation}
	for $\epsilon\in (0,\,\kappa_{10})$. This will imply from \eqref{eq:eq1046} that both $|c_3|$ and
	$|c_2|$ are of order $\sqrt{\epsilon}$ and hence by \eqref{c_0}, $|c_1|$ is also of order $\sqrt{\epsilon}$. Also, $0<1-\delta^{1/2}\,r_1$ implies $(\delta^{1/2}\,r_1)^n\to 0$ as $n\to\infty$. Therefore we obtain the desired bound
	\begin{equation*}
	\sup_{n\ge0}\sqrt{f_n}\le \kappa_{11}\,\sqrt{\epsilon},
	\end{equation*}
	for $\epsilon$ sufficiently small.
	The claim \eqref{claim} is established in the Appendix. The proof is now complete
	in view of \eqref{eq:eq1018}.
\end{proof}

We now complete the proof of Corollary \ref{cor:EScheme}.

\begin{proof}[Proof of Corollary \ref{cor:EScheme}]
	The first statement in the corollary is immediate from Corollary \ref{cor:UniformPOC_PE_2_1}
	and Theorem \ref{EScheme}.  For the second statement, we have from triangle inequality 
	$$\mathcal{W}_2(\mu_n^{N,\eps}, \mu_{n\eps}) \leq \mathcal{W}_2(\mu_n^{N,\eps}, \mu^N_{n\eps}) + \mathcal{W}_2(\mu^N_{n\eps},\,\mu_{n\epsilon}).$$
	Also, from Theorem \ref{EScheme},
	\begin{equation*} 
	\limsup_{N\to \infty} \sup_{n\ge 1}\E\mathcal{W}^2_2(\mu_n^{N,\eps}, \mu^N_{n\eps}) \le
	\limsup_{N\to \infty} \sup_{n\ge 1} \Big(\frac{1}{N}\sum_{i=1}^N \E|Y^{i,N,\epsilon}_n-X^{i,N}_{n\epsilon}|^2 \Big)\le C\eps.
	\end{equation*}
	The result now follows on combining the above two displays with Corollary \ref{cor:UniformPOC_PE_2_1}.
\end{proof}
\bigskip

\appendix
\section{Proof of \eqref{claim}.}
To see the first inequality in the claim \eqref{claim}, note that
\begin{align}\label{boundC_3}
\frac{c_3}{\sqrt{\epsilon}}
&\,=\frac{\delta^{1/2}\alpha\kappa_8}{\alpha\,\Big(\frac{\delta^{1/2}-\delta}{\epsilon}\Big)(1+\tilde\lambda\,\epsilon)\,-\,\tilde{C_2}}.
\end{align}
The inequality is now a consequence of the observation that $\delta^{1/2} \to 1$ as $\eps \to 0$ and 
$$
\lim_{\eps \to 0} \alpha\,\Big(\frac{\delta^{1/2}-\delta}{\epsilon}\Big)(1+\tilde\lambda\,\epsilon)\,-\,\tilde{C_2} = 
\alpha\,\tilde\lambda-\tilde{C_2} >0,
$$
where the last inequality is from Assumption \ref{A:A}.
Hence the first estimate holds.

The second inequality in the claim \eqref{claim} follows on observing that, as $\eps \to 0$,
%
$$\frac{(r_1-r_2)^2}{\eps^2}=\frac{p^2-4q}{\eps^2} = \frac{1}{\eps^2}\Big(1-\frac{e^{-\alpha\epsilon}}{\delta^{1/2}}\Big)^2+\frac{4A\,(1-e^{-\alpha\epsilon})}{\eps^2\alpha\,\delta^{3/2}}
\ge \frac{4A\,(1-e^{-\alpha\epsilon})}{\eps^2\alpha\,\delta^{3/2}} \to 4 \tilde C_2.
$$

For the third inequality, we need to show that $1-\delta^{1/2}\,r_1$ is of order at most $\sqrt{\epsilon}$. Clearly
\begin{align*}
1-\sqrt\delta r_1& = 1-\frac{1}{2}\Big[\sqrt{(p^2-4q)\delta}+p\sqrt\delta \Big]\nonumber\\
&= \frac{1}{2}\Big[(2-p\sqrt\delta )-\sqrt{(p^2-4q)\delta}\Big].\label{1-r1}
\end{align*}
Regarding $p,q$ and $\delta$ as functions of $\eps$, we see that
%
\begin{align*}
p &= p(\eps) = 2+ \eps p'(0) + O(\eps^2),\\
q&= q(\eps) = 1+ \eps q'(0) + O(\eps^2),\\
\sqrt\delta&= \sqrt{\delta(\eps)} = 1+ \frac{1}{2}\eps \delta'(0)+ O(\eps^2)
\end{align*}
Thus
\begin{align*}
(p^2-4q) &= 4\eps (p'(0)-q'(0)) + O(\eps^2) = O(\eps^2)
\end{align*}
where the last equality follows on checking that $p'(0)=q'(0)$.
This shows that $\sqrt{(p^2-4q)\delta} = O(\eps)$. 
Also, clearly $2-p\sqrt{\delta} = O(\eps)$ and so
$$(2-p\sqrt{\delta}) - \sqrt{(p^2-4q)\delta} = O(\eps).$$
The desired inequality  follows. \hfill \qed






\ACKNO{We would like to thank Tom Kurtz for many helpful discussions related to this work. We also thank the referee for a careful review.}


\end{document}